\documentclass[english]{article}

    \usepackage[nonatbib,preprint]{neurips_2023}

\usepackage{float}
\usepackage{algorithm,algorithmic}
\usepackage{amsmath}
\usepackage{amssymb}
\usepackage{graphicx}
\usepackage{bm}
\usepackage{amsthm}
\usepackage{subcaption}
\usepackage{enumitem}
\usepackage[title]{appendix}

\usepackage{lmodern}
\usepackage[utf8]{inputenc} 
\usepackage[T1]{fontenc}    
\usepackage{url}            
\usepackage{booktabs}       
\usepackage{amsfonts}       
\usepackage{nicefrac}       
\usepackage{microtype}  
\usepackage{mathrsfs}
\usepackage{epstopdf}

\numberwithin{equation}{section}
\usepackage{hyperref}
\hypersetup{
	colorlinks=true,
	linkcolor=red,
	filecolor=magenta,      
	citecolor=blue,
}
\newtheorem{lem}{Lemma}
\newtheorem{prop}{Proposition}

\newtheorem{defi}{Definition}

\numberwithin{theo}{section}
\numberwithin{lem}{section}
\numberwithin{prop}{section}
\numberwithin{alg}{section}
\numberwithin{assum}{section}
\numberwithin{defi}{section}
\numberwithin{coro}{section}

\theoremstyle{remark}
\newtheorem{remark}{Remarks}
\numberwithin{remark}{section}

\usepackage{stfloats}
\usepackage{wrapfig}

\usepackage{babel}

\usepackage{mathtools}
\def\Tiny{\fontsize{4pt}{4pt}\selectfont}
\newcommand*{\eqdef}{\ensuremath{\overset{\mathclap{\text{\Tiny def}}}{=}}}

\def\Tiny{\fontsize{4pt}{4pt}\selectfont}

\title{A Generic Closed-form Optimal Step-size for ADMM}
\date{}
\author{
	Yifan Ran,\,  Wei Dai \\
	Department of Electrical and Electronics Engineering\\
	Imperial College London\\
	\texttt{\{y.ran18,wei.dai1\}@imperial.ac.uk}\\}

\begin{document}
	\maketitle

	\begin{abstract}			
		In this work, we present a generic step-size choice for the ADMM type proximal algorithms. It admits a closed-form expression and is theoretically optimal with respect to a worst-case convergence rate bound. It is simply given by the ratio of Euclidean norms of the dual and primal solutions, i.e., $ ||\bm{\lambda}^\star|| / ||\bm{x}^\star||$. Numerical tests show that its practical performance is  near-optimal in general. The only challenge is that such a ratio is not known a priori and we provide two strategies to address it. The derivation of our step-size choice is based on studying the  fixed-point structure of ADMM using the proximal operator. However, we demonstrate that the classical proximal operator definition contains an input scaling issue. This leads to a scaled step-size optimization problem which would yield a false solution. 
		Such a issue is naturally avoided by our proposed new definition of the proximal operator. A series of its properties is established.
	\end{abstract}

	\section{Introduction}
	The choice of the step-size is well-known to have a significant impact on an algorithm's convergence rate. A well-chosen step-size significantly reduces computational time. This consideration is amplified as the size of data and models grow, resulting in large computational cost. Manual tuning of the step-size via brute force techniques, such as a grid search, adds additional overhead, and does not generally translate to alternative problem formulations or data sets. In this work, we establish a generic step-size choice for solving problem of the form $ f(\bm{x}) + g(\bm{x}) $ via Alternating  Direction Method of Multipliers (ADMM) \cite{boyd12,gabay1976dual,glowinski1975approximation}. The step-size choice is theoretically optimal and admits a closed-form expression as  $ ||\bm{\lambda}^\star|| / ||\bm{x}^\star||$. We will see that this simple expression can provide insights and helps with further processing.
	
	In practice, neither the primal nor the dual solutions of ADMM  are known a priori, i.e., we do not know the ratio $ ||\bm{\lambda}^\star|| / ||\bm{x}^\star||$  in advance. For practical use, we then provide two approaches to estimate such a ratio. The first approach is to directly estimate $ ||\bm{\lambda}^\star||$  and $||\bm{x}^\star|| $. In a nutshell, this approach is to estimate $ ||\bm{\lambda}^\star|| $  from the objective function and approximate $||\bm{x}^\star|| $   by exploiting data structure.
	We will carefully illustrate this method  in the context of Semidefinite Programs (SDPs) using two well-known applications --- the Boolean quadratic program \cite{poljak1995recipe, lemarechal1999semidefinite, malick2007spherical} and  super-resolution  \cite{candes4,tang2013compressed}. 
	The second approach is to adopt an adaptive step-size based on current iterates  $ ||\bm{\lambda}^k|| / ||\bm{x}^k||$ in a manner similar to \cite{9904868, 7738878}.

	Our step-size result is limited to ADMM type methods. ADMM is widely used and can be shown equivalent to some other popular methods, particularly the Douglas-Rachford Splitting (DRS) \cite{lions1979splitting, douglas1956numerical, peaceman1955numerical,  eckstein1992douglas}, see the equivalence proof in \cite{eckstein1992douglas,fortin2000augmented}. Recently, the Primal-Dual Hybrid Gradient (PDHG) method \cite{chambolle2011first, esser2010general, pock2009algorithm} has also been shown to be equivalent to ADMM in  \cite{o2020equivalence}.

	Our step-size result is obtained by optimizing a worst-case convergence rate bound of ADMM. There exists extensive work related to this rate bound, see some earlier work in \cite{monteiro2013iteration,he20121,he2015non,nishihara2015general}, and more recently in \cite{deng2016global,francca2016explicit,davis2016convergence,davis2017faster,hong2017linear,giselsson2017tight,moursi2019douglas, ryu2021large}. Among them, an elegant way of proving it is by appealing to  fixed-point theory, see e.g. \cite[Sect. 2.4]{ryu2021large}. In this work, we present a new proof that perhaps the simplest since it only requires two steps --- first proving a strictly decreasing error property, and then obtain the rate via a simple geometric sum formula. Another contribution of our proof is that we introduce  a $ 1/L $-cocoercive assumption that easily generalize the basic result. We show that the  standard rate is a special case of $ L=1 $, and when $ L \in (0,1) $ we have a sharp rate. 
	
	The fixed-point structure of ADMM is key to obtaining our step-size result. It is related to the proximal operator. 
	The proximal operator is a standard tool for nonsmooth, constrained, large-scale, or distributed optimization problems \cite{proxi_algs}.  From a data analysis point of view, it introduces a unified framework that relates many seemly disparate algorithms, see e.g. \cite{bauschke2017convex, combettes2005signal, ryu2016primer}. Many well-known algorithms can be written in terms of the proximal operator, and  are referred to as the proximal algorithms \cite{proxi_algs, rockafellar1976monotone}. ADMM can be classified as a primal-dual type proximal algorithm. 
	
	
	The main challenge we encounter is that the classical ADMM fixed-point expression $ \bm{x}^\star + \bm{\lambda}^\star /\gamma $ is in fact a scaled version, where $ \gamma $ denotes the step-size. The unscaled expression should be $ \sqrt{\gamma}\bm{x}^\star  + \bm{\lambda}^\star /\sqrt{\gamma} $. There exists a factor of $ \sqrt{\gamma} $ difference and we  demonstrate that it originates from the classical proximal operator definition. If we ignore such a scaling effect, we would obtain a scaled step-size optimization problem with a false solution. This issue is not obvious at all. We reveal it by introducing a new proximal operator definition. It is mathematically equivalent to the classical one, but naturally avoid such a scaling issue. 
	Moreover, we find that the new definition would yield a simpler Moreau decomposition expression than the classical one, which is also related to the scaling effect.
	We will establish a series of  properties of the new definition at the beginning of this paper.

	\textbf{Optimality:} Our proposed step-size choice is optimal in the sense that a worst-case convergence rate bound is optimized. It is an upper bound of the underlying actual convergence rate. Let us note that the actual convergence rate is intractable since it is characterized by the sum of all iterates $ \sum_{k} \Vert \bm{\psi}^{k+1} - \bm{\psi}^k \Vert^2 $. Therefore, the best we can expect is that our theoretical result from an upper bound is practically near-optimal. Numerical results show that our proposed step-size choice brings near-optimal practical performance in general. 
	
	\textbf{Limitations:} (i) The two approaches to estimate the theoretical optimal step-size have limitations.  
	The  current iterates estimator $ ||\bm{\lambda}^k|| / ||\bm{x}^k||$ works very nicely in practice and has roughly the same iteration number complexity performance as the theoretical optimal step-size choice. However,  due to that it is an adaptive step-size, caching factorizations is no longer possible (see a discussion in \cite[Sec. 4.2.3]{boyd12}). This may significantly increase the computational time of the ADMM for certain applications with large data. Also, it is an empirical approach and lack of theoretical analysis. Therefore, we do not have insights on the elements that affect the optimal step-size choice. 
	The other approach using direct estimation relies on a priori knowledge of either one of the two gradients   $ \nabla f, \nabla g $ to approximate $ \bm{\lambda}^\star$.  That said, this approach, at least currently, is limited to problems with at least one linear objective function, since in which case the gradient is a known constant. SDPs by definition meet this requirement. (ii) Our optimal step-size result is only valid for zero initialization, which is the most common choice for ADMM.
	
	\textbf{Key contributions:} (i) A new definition of proximal operator and its properties; (ii) A new generalized convergence rate and a simple proof;
	(iii) A novel generic closed-form optimal step-size for ADMM; (iv) Two approaches for estimating the optimal step-size choice.  		
	
	\textbf{Related work:} Step-size selection is an important topic and there exists extensive work in this field.  
	In \cite{nishihara2015general},  the authors  consider ADMM as a dynamic system. There, one first obtains a convergence rate bound by solving an SDP. Then, a grid search over the parameter space is  conducted to find the parameter choice that minimizes the rate bound. In \cite{xu2017adaptive, xu2017relaxed}, the authors adapt the successful Barzilai-Borwein spectral method for gradient descent and derive adaptive parameter choices for ADMM. 
	Recently,  \cite{ryu2020operator} proposed a general framework that computationally selects the optimal parameter by solving a series of SDPs.  
	In \cite{ghadimi2014optimal,teixeira2013optimal}, the authors  consider the quadratic program (QP), and the optimal ADMM parameters are derived under several specific settings.  \cite{shi2014linear} considers the decentralized consensus problems,  where strong assumptions, such as  Lipschitz continuity and the strong convexity  are required. 
	Our result does not require any strong assumption and admits a generically applicable closed-form expression. To our best knowledge, it is completely new in the literature.

	\section{New Proximal Operator}\label{sec_2} 
	Here, we first introduce a new  definition of the  proximal operator. It can be viewed as a symmetric definition of the classical one. The key difference is that it will not scale its input. This leads to a simpler Moreau decomposition expression and also avoids an unobvious fixed-point scaling issue when optimize the ADMM step-size.

	To start, recall the classical definition of the proximal operator 
	\begin{defi}\label{prox_lam}
		Given a CCP function $ f $, the classical proximal operator with a scalar parameter $\gamma > 0 $ is defined as
		\begin{equation}\label{lam_prox}
			\textbf{Prox}_{\gamma f}(\cdot) \eqdef  \underset{\bm{z}}{\text{argmin}} \,\,  \gamma f(\bm{z}) + \frac{1}{2}\Vert \bm{z} - \cdot\Vert^2
			= \underset{\bm{z}}{\text{argmin}} \,\,   f(\bm{z}) + \frac{1}{2\gamma}\Vert \bm{z} - \cdot\Vert^2,
		\end{equation}
		where $ \Vert \cdot \Vert $ denotes the Euclidean norm.
	\end{defi}	
	
	We propose the following new definition:
	\begin{defi}\label{new_def1}
		Given a CCP function $ f $  and a  parameter $\rho \neq 0 $, we propose the following new definition:
		\begin{equation}\label{def_right}
			\textbf{Prox}_{f \rho} (\cdot) \eqdef \underset{\bm{z}}{\text{argmin}} \,\, f(\rho\bm{z} )+ \frac{1}{2}\Vert \bm{z} - \cdot\Vert^2
			=\frac{1}{\rho}\,\underset{\bm{z}}{\text{argmin}} \,\, f(\bm{z})+ \frac{1}{2}\Vert \frac{\bm{z} }{\rho}- \cdot\Vert^2.
		\end{equation}
	\end{defi}
	From an operator theory view, the new definition uses a right-multiplication $ f \rho = f(\rho\,\cdot ) $ and the classical one uses a  left-multiplication $ \gamma f = \gamma f(\cdot ) $. In this sense, one may view them as a pair of symmetric definitions.
	
	The key difference of the above two definitions  is whether the input, denoted by $\cdot$, is scaled or not. To see this, first rewrite the classical definition into 
	$  \text{argmin} \, f(\bm{z}) + \frac{1}{2}\Vert (\bm{z} - \cdot)/\sqrt{\gamma}\Vert^2$. Clearly, the input is scaled by a factor of $ \sqrt{\gamma} $. On the other hand, for the new definition $ 1/\rho \,\,{\text{argmin}} \, f(\bm{z})+ \frac{1}{2}\Vert {\bm{z} }/{\rho}- \cdot\Vert^2$, the input is clearly not scaled. The existence of the scaling  can also be seen from the following translation rule:	
	\begin{lem}\label{lem_rel}
		The two  types of proximal operators above are related through
		\begin{equation}\label{eq_rel}
			\textbf{Prox}_{f \rho}(\bm{v}) = \frac{1}{\rho}\textbf{Prox}_{\rho^2f} (  \rho{\bm{v} } ),\quad
			\textbf{Prox}_{\gamma f}(\bm{v}) = \sqrt{\gamma}\textbf{Prox}_{f\sqrt{\gamma}} ( \frac{\bm{v}}{\sqrt{\gamma}} ).
		\end{equation} 		
	\end{lem}
	
	For concreteness, we establish the following property for the new definition:
	\begin{prop}\label{prop_firm}
		$\textbf{Prox}_{f\rho}$ as defined in \eqref{def_right} is firmly nonexpansive and single-valued.	
	\end{prop}
	
	\subsection{Generalization to an operator parameter}
	The parameter of the proximal operator is also the step-size for ADMM and it plays a key role in the convergence rate issue.	
	It is possible to use a generalized parameter, which in principle can provide faster convergence rate than the scalar case.
	
	For generality, we discuss an operator parameter, which can be reduced to a matrix case.
	We start with the classical definition.
	\begin{defi}
		Given a CCP function $ f $ with a positive definite operator $ \mathcal{M} \succ 0  $, the classical proximal operator can be defined as
		\begin{equation}\label{H_prox}
			\textbf{Prox}_{\mathcal{M} f} (\cdot) \eqdef {\text{argmin}} \,\, f(\bm{z}) + \frac{1}{2}\Vert \bm{z} - \cdot\Vert^2_{\mathcal{M}^{-1}} ,
		\end{equation} 
		where	$ \Vert \cdot \Vert_{\mathcal{A}}  $ is a norm induced by the inner product $ \langle \cdot\,,\mathcal{A} \,\cdot \rangle$ with $ \mathcal{A} \succ 0  $.
	\end{defi}
	
	We propose the following new definition:
	\begin{defi}\label{new_def2}
		Given a CCP function $ f $ and a   bijective, linear and continuous operator $ \mathcal{S} $, we propose the following new definition:
		\begin{equation}\label{def_new}
			\textbf{Prox}_{f\mathcal{S}} (\cdot) \eqdef \underset{\bm{z}}{\text{argmin}} \,\, f(\mathcal{S}\bm{z})+ \frac{1}{2}\Vert \bm{z} - \cdot\Vert^2
			=  \mathcal{S}^{-1} \underset{\bm{z}}{\text{argmin}} \,\, f(\bm{z})+ \frac{1}{2}\Vert \mathcal{S}^{-1}\bm{z} - \cdot\Vert^2.
		\end{equation}
	\end{defi}
	The translation rule of the new and classical proximal operators is given by the following lemma.
	\begin{lem}\label{lem_rel2}
		Given a bijective, linear and continuous operator $ \mathcal{S}$, let $ \mathcal{M} \eqdef \mathcal{S}\mathcal{S}^* $. The classical and new proximal operators are related through
		\begin{equation}
			\textbf{Prox}_{f\mathcal{S}}( \bm{v} )  = \mathcal{S}^{-1} \textbf{Prox}_{\mathcal{M} f}(\mathcal{S}\bm{v}),\quad
			\textbf{Prox}_{\mathcal{M} f}(\bm{v}) = \mathcal{S}\textbf{Prox}_{f\mathcal{S}}( \mathcal{S}^{-1}\bm{v} ), 
		\end{equation} 	
		where 	$ \mathcal{S}^* $ denotes the adjoint operator of $ \mathcal{S} $.
	\end{lem}
	The new definition admits the following property: 
	\begin{prop}\label{prop_firm2}
		$\textbf{Prox}_{f\mathcal{S}}$ as defined in \eqref{def_new} is firmly nonexpansive and single-valued.	
	\end{prop}

	\subsection{An evaluation result}
	While a generalized parameter  could achieve faster convergence than a scalar parameter, its corresponding proximal operator is often hard to evaluate. Nevertheless, it is possible to find some special cases. Below, we propose  a  closed-form evaluation result.
	
	\begin{prop}\label{prop_evaluation}
		Let $\bm{P}$ be a full-rank square matrix with orthogonal columns. Then, 
		\begin{equation}
			\bm{P}^\text{T}\mathcal{T} ((\bm{P}^{-1})^\text{T}\bm{v}) 
			=  \underset{\bm{x}}{\text{argmin}} \,\, \Vert \bm{P}\bm{x} \Vert_1 + \frac{1}{2}\Vert\bm{x} - \bm{v} \Vert^2  = \textbf{Prox}_{f\bm{P}}(\bm{v}),
		\end{equation}
		where $ \mathcal{T}  $ is a soft thresholding operator,  defined as
		\begin{equation*}
			\mathcal{T} (\bm{u}) \eqdef \text{sgn}(\bm{u}) \odot \{|\bm{u}| -\bm{1}\}_+, 
		\end{equation*}	
		and where $ |\cdot| $ denotes the absolute value; $ \bm{1} $ denotes a ones vector; $ \{\cdot\}_+ $ is the operation of setting all negative entries to zeros.
	\end{prop}

	\subsection{Simplified Moreau decomposition}
	The following relation is known as Moreau decomposition:
	\begin{equation}\label{moreau_basic}
		\bm{v} = \textbf{Prox}_{f}(\bm{v}) + \textbf{Prox}_{f^*} (\bm{v}),
	\end{equation}
	where $ f^*(\cdot) \eqdef \sup\, \langle \bm{z}, \cdot \rangle - f(\bm{z})$ is the convex conjugate of $ f $. The above relation always holds and is the main relationship between proximal operators and duality \cite{proxi_algs}. 
	
	Its extension via the classical proximal operator is given by
	\begin{equation}\label{moreau_eq1}
		\bm{v} = \textbf{Prox}_{\gamma f}(\bm{v}) + \gamma\textbf{Prox}_{\frac{1}{\gamma}f^*} (\frac{1}{\gamma}\bm{v}).
	\end{equation}	
	Our proposed new proximal operator would yield
	\begin{equation}\label{moreau_eq2}
		\bm{v} = \textbf{Prox}_{f\rho} (\bm{v}) +  \textbf{Prox}_{f^*\frac{1}{\rho}} (\bm{v}).
	\end{equation}
	We see that the expression in \eqref{moreau_eq2} is simpler than the one in \eqref{moreau_eq1}. Moreover, there exists symmetry between the two components in  \eqref{moreau_basic}. Such a symmetry is lost in \eqref{moreau_eq1}, while kept in \eqref{moreau_eq2}. In this sense, we argue that \eqref{moreau_eq2} is a more natural extension than \eqref{moreau_eq1}.

	\section{ADMM Algorithm}

	ADMM is a typical proximal algorithm. Here, we consider problem $   f(\bm{x}) + g(\bm{x}) $. 	
	To solve it using ADMM,  we first rewrite it into a constrained form as
	\begin{align}\label{opt_problem}
		\underset{\bm{x},\bm{z}}{\text{minimize}}\,\,\quad &   f(\bm{x}) + g(\bm{z}),\nonumber\\
		\text{subject\,to}\:\quad & \:\,\,\,\bm{x} - \bm{z} = 0.
	\end{align}
	One can use either one of the following augmented Lagrange functions, which correspond to the classical and the new proximal operators:
	\begin{equation}\label{lagrangian_1}
		\mathcal{L}_\gamma(\bm{x},\bm{z};\bm{\lambda}) 
		\,\eqdef\,  f(\bm{x}) + g(\bm{z}) + \langle \bm{\lambda}, \bm{x} - \bm{z}\rangle + \frac{\gamma}{2}\Vert\bm{x} - \bm{z}\Vert^2, \quad \gamma > 0. 
	\end{equation}	
	\begin{equation}
		\mathcal{L}_\rho(\bm{x},\bm{z};\bm{\lambda}) 
		\,\eqdef\,  f(\bm{x}) + g(\bm{z}) + \langle \bm{\lambda}, \bm{x} - \bm{z}\rangle + \frac{1}{2}\Vert\rho (\bm{x} - \bm{z})\Vert^2, \,\, \rho \neq 0. 
	\end{equation}
	For a better connection to the current literature, we use the first Lagrangian with step-size $\gamma>0$ throughout the rest of the paper.
	The ADMM primal iterates $ \bm{x} $ and $ \bm{z} $ are obtained by minimizing $ \mathcal{L}_\gamma $ and the dual iterate $ \bm{\lambda} $ is obtained via a dual ascent manner. Specifically, the ADMM iterates are given by
	\begin{align*}
		\text{Classi}& \text{cal iterates} \hspace{15em}   \text{New iterates} \\
		\bm{x}^{k+1} =\,\,&\textbf{Prox}_{\frac{1}{\gamma}f}(\bm{z}^{k}-\bm{\lambda}^{k}/\gamma),			\quad\quad\quad\quad\quad	
		{\bm{x}}^{k+1} =\,\,  \frac{1}{\sqrt{\gamma}}\textbf{Prox}_{f\frac{1}{\sqrt{\gamma}}}(\sqrt{\gamma}{\bm{z}}^{k}-{\bm{\lambda}}^{k}/\sqrt{\gamma}),\\				
		\bm{z}^{k+1} =\,\, &\textbf{Prox}_{\frac{1}{\gamma}g}(\bm{x}^{k+1} + \bm{\lambda}^{k}/\gamma), 	\quad\Longleftrightarrow\quad 	
		{\bm{z}}^{k+1} =\,\,  \frac{1}{\sqrt{\gamma}}\textbf{Prox}_{g\frac{1}{\sqrt{\gamma}}}(\sqrt{\gamma}{\bm{x}}^{k+1} + {\bm{\lambda}}^{k}/\sqrt{\gamma}),\\
		\bm{\lambda}^{k+1} =\,\,  &\bm{\lambda}^{k} + \gamma(\bm{x}^{k+1} - \bm{z}^{k+1}) 				 \quad\quad\quad\quad\quad\,
		{\bm{\lambda}}^{k+1} =\,\,  {\bm{\lambda}}^{k} + \gamma({\bm{x}}^{k+1} - {\bm{z}}^{k+1})\\
		&\big\downarrow \hspace{22em} \big\downarrow\\
		\text{Fixed-point}&\,\, \text{Characterization} \hspace{9em}   \text{Fixed-point Characterization} \\
		\bm{\psi}^{k+1} \eqdef \,&\bm{x}^{k+1} + \bm{\lambda}^{k}/\gamma= \mathcal{F}_1\bm{\psi}^{k}	 \hspace{6em}
		\bar{\bm{\psi}}^{k+1} \eqdef \sqrt{\gamma}{\bm{x}}^{k+1} + {\bm{\lambda}}^{k}/\sqrt{\gamma} = \mathcal{F}_2\bar{\bm{\psi}}^{k}	
	\end{align*}
	where $ \mathcal{F}_1 \eqdef  \frac{1}{2}\mathcal{I} + \frac{1}{2}(2\textbf{Prox}_{\frac{1}{\gamma}f}- \mathcal{I}) \circ (2\textbf{Prox}_{\frac{1}{\gamma}g} - \mathcal{I}) $, \,
	$ \mathcal{F}_2 \eqdef  \frac{1}{2}\mathcal{I} + \frac{1}{2}(2\textbf{Prox}_{f\frac{1}{\sqrt{\gamma}}}- \mathcal{I}) \circ (2\textbf{Prox}_{g\frac{1}{\sqrt{\gamma}}} - \mathcal{I}) $.
	
	We see that for variables $ \bm{x,z,\psi} $, the new iterates have proximal inputs of the form  $ \sqrt{\gamma}\,(\cdot) \pm (\cdot)/\sqrt{\gamma}$, while the classical one gives $(\cdot) \pm (\cdot)/{\gamma}$. This difference is the scaling issue that will be carefully discussed in the next section.

	\section{Optimal step-size}
	
	Here, we present a theoretically optimal  step-size  obtained by minimizing a worst-case convergence rate bound. The convergence rate is characterized by the ADMM fixed-point $\bm{\psi}$. First, we propose the following rate characterization:
	\begin{prop}\label{prop_convergence}
		Suppose an operator $ \mathcal{F} $  is firmly nonexpansive. For a set of fixed-point iterates $ \{\bm{\psi}^k\} $  that converges to  $\bm{\psi}^\star\in \text{Fix} (\mathcal{F}) $, it admits a non-ergodic worst-case convergence rate as
		\begin{equation}\label{bd1}
			\Vert \bm{\psi}^{k+1} - \bm{\psi}^k  \Vert^2 \leq \frac{1}{k+1} \Vert \bm{\psi}^\star - \bm{\psi}^0  \Vert^2.
		\end{equation}
		Furthermore, suppose $ \mathcal{F} $ is $1/L$-cocoercive with  $ L \in (0,1) $. Then, the above rate is improved to 	
		\begin{equation}\label{bd2}
			\Vert \bm{\psi}^{k+1} - \bm{\psi}^k  \Vert^2 \leq \big(\frac{L}{2-L}\big)^k \cdot\frac{ 1 - \frac{L}{2-L}}{1 -(\frac{L}{2-L})^{k+1} }\Vert \bm{\psi}^\star - \bm{\psi}^0  \Vert^2.
		\end{equation}
	\end{prop} 
	Let us note that \eqref{bd1} is a special case of $ L=1 $. To our knowledge, such a generalization using $1/L$-cocoercive assumption is new in the literature. Also, our proofs of   \eqref{bd1} and \eqref{bd2} are new and simple.

	The initialization $ \bm{\psi}^0  $ can be arbitrarily chosen, but the convergence rate can also be arbitrarily slow. The default choice is $ \bm{\psi}^0 = 0$ and we will limit our discussion to this case. 
	
	Since different $\gamma$ yields different set of fixed-point iterates $ \{\bm{\psi}^k\}  $,  we aim to find the iterates that give the smallest upper bound. Clearly, the iteration number $ k $ and Lipschitz constant $ L $ related terms are constant when we optimize $\gamma$. That said, to achieve the purpose, we only need to minimize $ \Vert\bm{\psi}^\star\Vert^2 $ w.r.t. $\gamma$.
	
	\subsection{The failure of the classical definition}
	As shown in previous section, the classical proximal operator gives an ADMM fixed-point as $  \bm{\psi}^\star = \bm{x}^\star+ \bm{\lambda}^\star/ \gamma$. If we directly minimize its Euclidean norm, we obtain 
	\begin{equation}\label{prob1}
		\underset{\gamma>0}{\text{minimize}} \,\: \Vert \bm{x}^\star+ \bm{\lambda}^\star/\gamma \Vert^2
		\,\,= \underset{\gamma>0}{\text{minimize}} \,\: \frac{1}{\gamma^2}\Vert\bm{\lambda}^\star \Vert^2 + \frac{2}{\gamma}\langle \bm{x}^\star, \bm{\lambda}^\star \rangle.
	\end{equation}
	Let us note that the objective function above is non-convex in nature. To our knowledge, there is no general closed-form solution available (unless  the sign of $ \langle \bm{x}^\star, \bm{\lambda}^\star \rangle $ is  negative, which cannot be guaranteed).
	
	Alternatively, suppose we redefine the augmented term of the ADMM Lagrangian as in \eqref{lagrangian_1} into $ \frac{1}{2\gamma}\Vert\bm{x} - \bm{z}\Vert^2 $. Then, similar to the above process, we obtain a convex optimization problem
	\begin{equation}\label{prob2}
		\underset{\gamma>0}{\text{minimize}} \,\: \Vert \bm{x}^\star+ \gamma\bm{\lambda}^\star \Vert^2
		\,\,= \underset{\gamma>0}{\text{minimize}} \,\: {\gamma^2}\Vert\bm{\lambda}^\star \Vert^2 + 2{\gamma}\langle \bm{x}^\star, \bm{\lambda}^\star \rangle,
	\end{equation}
	Ignore $ \gamma>0 $ for now and we would obtain
	\begin{equation}
		\gamma^\star = \frac{-\langle \bm{x}^\star, \bm{\lambda}^\star \rangle}{\Vert\bm{\lambda}^\star \Vert^2}.
	\end{equation}
	Let us note that the numerator  $ -\langle \bm{x}^\star, \bm{\lambda}^\star \rangle $ can be positive, negative or even zero (the zero case holds for all SDPs, see appendix). For the  negative and zero cases, to enforce constraint $ \gamma>0 $,  a feasible solution would be $ 0 + \epsilon $ where $ \epsilon $ is an infinitely small positive number. That said, the optimal step-size should be chosen to be infinitely small in absolute value. This is clearly against any practical simulation.
	
	Also, problem \eqref{prob1} and \eqref{prob2} does not share the same optimal solution in general. This should not happen since there is no fundamental difference between the two definitions $ \frac{\gamma}{2}\Vert\cdot\Vert^2 $ and $ \frac{1}{2\gamma}\Vert\cdot\Vert^2  $. Therefore, we conclude that something must be wrong. 
	
	The above issues will be naturally avoided by using our proposed new proximal operator. When we compare the classical and the new definitions, we realize the above issues come from the input scaling effect. 
	Specifically, by definition $  \textbf{Prox}_{\frac{1}{\gamma}f}(\cdot)  =  {\text{argmin}} \,\,   f(\bm{z}) + \frac{1}{2}\Vert (\bm{z} - \cdot)/\sqrt{\gamma}\Vert^2$, which implies that the input is scaled into $ (\cdot)/\sqrt{\gamma} $. That said, we are actually dealing with  $ \Vert\bm{\psi}^\star/\sqrt{\gamma}\Vert^2 $ in \eqref{prob1}, and  $ \Vert\sqrt{\gamma}\bm{\psi}^\star\Vert^2 $ in \eqref{prob2}. Such extra scaling factors are introduced by the classical proximal operator definition and has nothing to do with ADMM.
	We should remove it when finding the ADMM optimal step-size. This yields  $ \Vert  \sqrt{\gamma}{\bm{x}}^\star + {\bm{\lambda}}^\star/\sqrt{\gamma} \Vert^2 $, which is exactly what happens when we use the new proximal operator as in the following section.

	\subsection{The optimal step-size}
	The new proximal operator gives an ADMM fixed-point as $ \sqrt{\gamma}{\bm{x}}^\star + {\bm{\lambda}}^\star/\sqrt{\gamma}$. This gives
	\begin{equation}
		\underset{\gamma}{\text{minimize}}  \,\: \Vert  \sqrt{\gamma}{\bm{x}}^\star + {\bm{\lambda}}^\star/\sqrt{\gamma} \Vert^2
		\,\,= \underset{\gamma>0}{\text{minimize}} \,\: \gamma\Vert\bm{x}^\star \Vert^2 + \frac{1}{\gamma}\Vert\bm{\lambda}^\star\Vert^2,
	\end{equation}
	where the minimum is obtained if and only if $ \gamma\Vert\bm{x}^\star \Vert^2 = \frac{1}{\gamma}\Vert\bm{\lambda}^\star\Vert^2 $. This gives
	\begin{equation}\label{scalar_choice}
		\gamma^\star = \frac{\Vert\bm{\lambda}^\star\Vert}{\Vert\bm{x}^\star\Vert}. 
	\end{equation}
	Clearly, the above $ \gamma^\star $ naturally meets the positivity requirement and is therefore a feasible solution. 
	\subsection{Estimation}
	In general, the optimal step-size result in \eqref{scalar_choice} is not a priori knowledge and we provide two ways to handle it. First is to adopt an adaptive step-size based on current iterates  $ ||\bm{\lambda}^k|| / ||\bm{x}^k||$ in a manner similar to \cite{9904868, 7738878}. 
	The other approach is to directly estimate $ \Vert\bm{\lambda}^\star\Vert$ and $\Vert\bm{x}^\star\Vert $, which will be the key content of the later section on pre-fixed step-size. 
	It uses the following generic result:
	\begin{lem}\label{prop_gradient}
		Let $ (\bm{x}^\star, \bm{\lambda}^\star)$ be the  primal-dual optimal solution pair for   ADMM solving problem $ f(\bm{x}) + g (\bm{x}) $. We have
		\begin{equation}
			\bm{\lambda}^{\star} = -\widetilde{\nabla} f (\bm{x}^\star) = \widetilde{\nabla} g (\bm{x}^\star),
		\end{equation}
		where  $ \widetilde{\nabla}  $ denotes the actual choice of gradient.
	\end{lem}

	\subsection{Generalization to f(Ax) + g(x)}
	The optimal step-size result in \eqref{scalar_choice} can be easily generalized to problem $ f(\bm{Ax}) + g(\bm{x}) $. To extend, all we need to ask is how the ADMM fixed-point changes. 
	Suppose we do variable substitution via $ \bm{z} \eqdef \bm{Ax}$ (recall \eqref{opt_problem}), then  the ADMM fixed-point is given by $ \bm{\psi}^\star =  \sqrt{\gamma}{\bm{Ax}}^\star + {\bm{\lambda}}^\star/\sqrt{\gamma} $. It follows that  
	\begin{equation}
		\gamma^\star = \frac{\Vert\bm{\lambda}^\star\Vert}{\Vert\bm{Ax}^\star\Vert}. 
	\end{equation}
	Other extensions are straightforward by viewing the ADMM fixed-point structure.

	
	\subsection{Generalization to a matrix step-size}
	The optimal step-size result in \eqref{scalar_choice} can be generalized to a matrix case, which would provide better performance in principle. 
	\begin{prop}\label{prop_matrix_step}
		Let $ \bm{M} \eqdef \text{Diag}(\bm{d})$ be the step-size, where $ \bm{d} $ is a positive vector and  $ d_i $ denotes its $ i $-th element.
		We have
		\begin{equation}\label{m_step}
			\bm{M}^\star = \text{Diag}(\bm{d}^\star), \qquad
			d_i^\star = \begin{cases}
				\vert\lambda_i^\star/x_i^\star\vert  &  \quad{x}_i^\star \neq 0, \\
				+\infty & 	\quad{x}_i^\star = 0,
			\end{cases} 
		\end{equation}
		where $ {x}_i^\star $ and $ \lambda_i^\star $ are the $ i $-th entries of the primal and dual optimal solutions, respectively, and where $ \text{Diag}(\cdot) $ denotes creating a diagonal matrix from an input vector.
	\end{prop}
	

	\section{Pre-fixed step-sizes for SDPs}
	Semidefinite program (SDP) considers problem: $ {\text{minimize}}\,  f(\bm{X}), s.t.  \,\bm{X}\succeq 0$, where $ f $ is linear, and where $ \bm{X} \in \mathbb{R}^{N\times N} $ is symmetric.
	It can be written into a $ f + g $ form as 
	\begin{equation}\label{sdp_pro}
		\text{minimize}\,\,  f(\bm{X}) + \delta_{\mathbb{S}_+^N}(\bm{X}).
	\end{equation}
	where $ \delta_{\mathbb{S}_+^N}  $  is an indicator function w.r.t.   the N-dimensional positive semidefinite cone $ \mathbb{S}_+^N $.
	For SDP, the optimal step-size in \eqref{scalar_choice} is given in a matrix form as $ \gamma^\star = {\Vert\bm{\Lambda}^\star\Vert}/{\Vert\bm{X}^\star\Vert} $. We will show that such a $ \gamma^\star $ can be well approximated into a pre-fixed choice in the context of SDP. Particularly, SDP is known to have an extremely high computational cost, and therefore our result could be of great practical interests. Below, we illustrate our approximation scheme through two well-known applications.

	\subsection{Boolean quadratic program}\label{bqp}
	Boolean quadratic program (BQP)  considers problem: $ {\text{minimize}}\, \Vert \bm{A}\bm{x} - \bm{b}  \Vert^2, s.t.  \,{x}_i^2 = 1, \, \forall i$. 
	However, this formulation is NP-hard. A common practice is to translate it into an SDP relaxation, which can be written in a compact form as	
	\begin{align}\label{r_BQP}
		\underset{\bm{X}}{\text{minimize}}\quad  & \langle \bm{X}, \bm{C} \rangle + \delta_{\mathbb{S}^{N+1}_+}(\bm{X})  \nonumber\\
		\text{subject\,to}\quad 
		&  \text{diag}(\bm{X})= \bm{1}_{(N+1)\times 1}, 
	\end{align}
	where $ \bm{X} \eqdef  \left[\begin{array}{cc} \bm{X}_1 & \bm{x}\\ \bm{x}^\text{T}  & 1 \end{array} \right]$, 
	$ \bm{C} \eqdef  \left[\begin{array}{cc} \bm{A}^\text{T}\bm{A} & -\bm{A}^\text{T}\bm{b} \\ -\bm{b}^\text{T}\bm{A}  & 0 \end{array} \right]$, 
	$\bm{X}_1\in \mathbb{R}^{N\times N} , \bm{x} \in \mathbb{R}^{N}, \bm{A} \in \mathbb{R}^{K\times N} $, $ \bm{b} \in \mathbb{R}^{K} $.
	
	We aim to estimate the optimal step-size choice $ \gamma^\star =  \Vert\bm{\Lambda}^\star\Vert/ \Vert \bm{X}^\star \Vert $. First, we have 
	\begin{equation}
		\sqrt{N+1} \leq \Vert\bm{X}^\star\Vert \leq N+1,
	\end{equation}
	where the upper bound is attained when $ \bm{X}^\star $ is rank-1 (in which case $ \bm{X}_1^\star = \bm{x}\bm{x}^\text{T}$ and the SDP relaxation is tight), and the lower bound is attained when  $ \bm{X}^\star $ has $ (N+1) $ equal  eigenvalues.

	For $ \bm{\Lambda}^\star $, standard Lagrange analysis gives $ \bm{\Lambda}^{\star} = -\bm{C} - \text{Diag}(\bm{\mu}^\star) $ (see appendix). 
	Empirically, we find  that  $ \text{Diag}(\bm{\mu}^\star) $ is not a dominant term, and one may use $ {\Vert\bm{C}\Vert}$ to replace the unknown $ \Vert\bm{\Lambda}^\star\Vert $. 
	We arrive at
	\begin{equation}\label{prefix_bqp}
		\gamma_{prefix} = \frac{\Vert\bm{C}\Vert}{d}, 
	\end{equation}
	where $ d \in [\sqrt{N+1}, N+1] $. 
	
	
	Let us note that the above SDP formulation is a relaxation of the original BQP. That said, if the relaxation is too loose, the result is almost meaningless to the original BQP. Since the upper bound of $ d $ stands for tight relaxation, we can directly choose $ d $ as $ N+1 $ for simplicity. Unless one intentionally deals with a very loose SDP relaxation,  in which case one may choose $ d  = \sqrt{N+1} $. Otherwise, 
	$ \gamma_{prefix} = {\Vert\bm{C}\Vert}/(N+1)$ should be a  good choice in general.

	\subsection{Super-resolution}\label{sr}
	
	Super-resolution (SR) problem deals with  model  $ x_n^\star = \sum_{k}c_{k}e^{-i2\pi n\tau_{k}}, \: n =0, 1,\dots, N-1 $. The available data  is  a subset of the collection $ \{ x_n^\star \} $ denoted by  $ \{x_j^\star\}_{j\in\Omega} $.
	The problem can be written into the following compact SDP formulation:
	\begin{align}
		\underset{\bm{X}}{\text{minimize}}\quad  & \langle \bm{X}, \bm{C} \rangle  + \delta_{\mathbb{S}^{N+1}_+}(\bm{X})   \nonumber\\
		\text{subject\,to}\quad 
		&  x_j =  x_j^\star, \: \forall j\in\Omega  \nonumber
	\end{align}
	where $ \bm{X} \eqdef \left[\begin{array}{cc}\mathcal{T}(\bm{u}) & \bm{x}\\\bm{x}^\text{H} & t\end{array}\right]$, 
	$ \bm{C} \eqdef  \left[\begin{array}{cc} \frac{1}{2N}\bm{I}_N & 0 \\ 0  &  \frac{1}{2} \end{array} \right]$, $\mathcal{T}(\bm{u}) \in \mathbb{R}^{N\times N} $,  $ \bm{u},\bm{x} \in \mathbb{R}^{N} $, $ t \in \mathbb{R} $,
	and where $ \bm{I}_N $ denotes an identity matrix in $ \mathbb{R}^{N\times N} $, $ \mathcal{T} $ is a Toeplitz operator that maps an input vector to a Toeplitz matrix.
	
	We aim to estimate the optimal step-size choice  $ \gamma^\star =  \Vert\bm{\Lambda}^\star\Vert/ \Vert \bm{X}^\star \Vert $. First, we have
	\begin{equation}
		(N+1)\sqrt{K} m_{avg}  \le \Vert\bm{X}^\star\Vert < (N+1)K m_{avg} ,
	\end{equation}
	where $ m_{avg} \eqdef \frac{1}{K}\sum_{k=1}^K|c_{k}| $ denotes the average magnitude of the spikes $ \bm{c} $. The  upper bound corresponds to $ \bm{X}^\star $ has only one eigenvalue, which is not attainable since by definition rank$(\bm{X}^\star)=K$. The lower bound is obtained when $ \bm{X}^\star $ has $ K $ equal eigenvalues.

	For $ \bm{\Lambda}^\star $, standard Lagrange analysis gives  $ \bm{\Lambda}^{\star} = -\bm{C} - \bm{M}_{\Omega} $ (see appendix). 
	Empirically, we find  that  $ \bm{M}_{\Omega} $ is not a dominant term, and one may use $ {\Vert\bm{C}\Vert}$ to replace the unknown $ \Vert\bm{\Lambda}^\star\Vert $. 
	We arrive at
	\begin{equation}\label{prefix_sr}
		\gamma_{prefix} =\frac{\Vert \bm{C} \Vert}{d m_{avg}},
	\end{equation}
	where $ d \in [(N+1)\sqrt{K}, (N+1)K) $.

	Suppose the elements of the spike amplitude $ \bm{c} $ is randomly generated from i.i.d. Gaussian distribution  $ \mathcal{N}(0,\sigma) $. Then,  its absolute value $ |c| $  follows the folded normal distribution \cite{leone1961folded} with  mean value  $ \sigma\sqrt{2/\pi} $. Therefore, we recommend setting $ m_{avg} = \sigma\sqrt{2/\pi} $ in this case. Numerically, we find that the eigenvalue distribution of $ \bm{X} $ are not concentrated in our tests. This suggest choosing the upper bound of $ d $ is not a good idea. 
	For simplicity, we then directly use the lower bound $ d  = (N+1)\sqrt{K} $. This gives  $ \gamma_{prefix} = {\Vert\bm{C}\Vert}/((N+1)\sqrt{K}\sigma\sqrt{2/\pi})$.

	\section{Numerical results}
	In this section we evaluate the practical performance of different choices of step-sizes. We would compare these choices to the underlying best choice, which is obtained by exhaustive search. The performance is evaluated by the iteration number complexity. We would stop the algorithm and record its iteration number if a mean squared error threshold of  $ 1 \times 10^{-10} $ is reached. We calculate the error using the `ground truth' generated from the solver CVX \cite{grant11} in a best precision mode.
	
	Apart from the two SDP applications from previous section, we also test another 4 popular applications adopted from \cite{boyd12}. We will consider (i) Lasso: minimize $  1/2|| \bm{Ax - b} ||_2^2 + \alpha || \bm{x} ||_1 $;
	(ii) Least absolute deviations (LAD): minimize  $ 1/2|| \bm{Ax} - \bm{b} ||_1 $; (iii) Quadratic programming (QP): minimize $ 1/2 || \bm{Ax - b} ||_2^2$,
	s.t.  $a \leq \bm{x} \leq b $; (iv) Total Variation (TV): minimize  $ (1/2)||\bm{x - b}||_2^2 + \alpha \sum_i |x_{i+1} - x_i| $.

	\begin{figure}[H]
		\begin{center}
			\begin{subfigure}{0.3\textwidth}
				\includegraphics[scale=0.25]{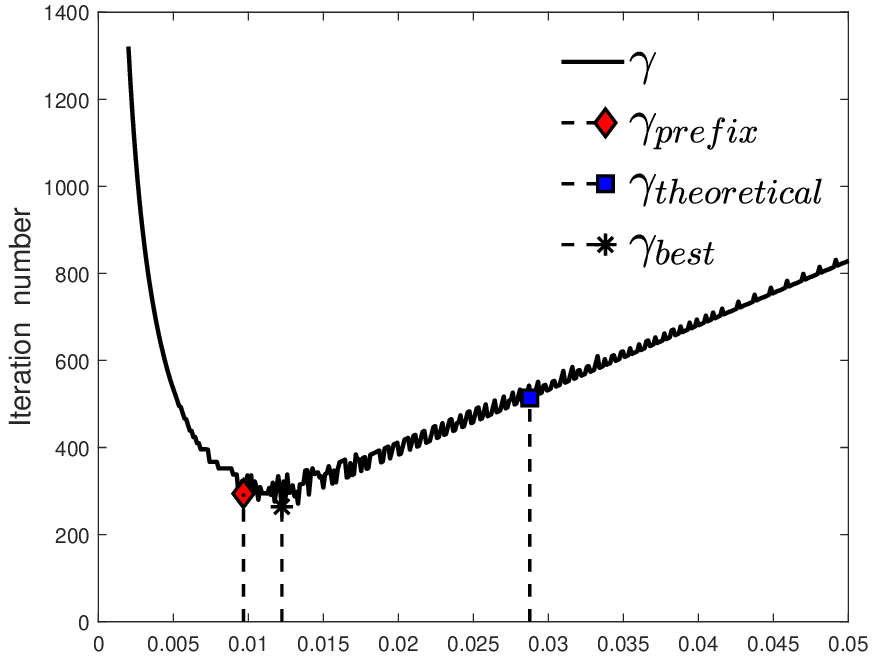} 	
				\caption{BQP}
				\label{1a}	
			\end{subfigure}
			\begin{subfigure}{0.3\textwidth}	
				\includegraphics[scale=0.25]{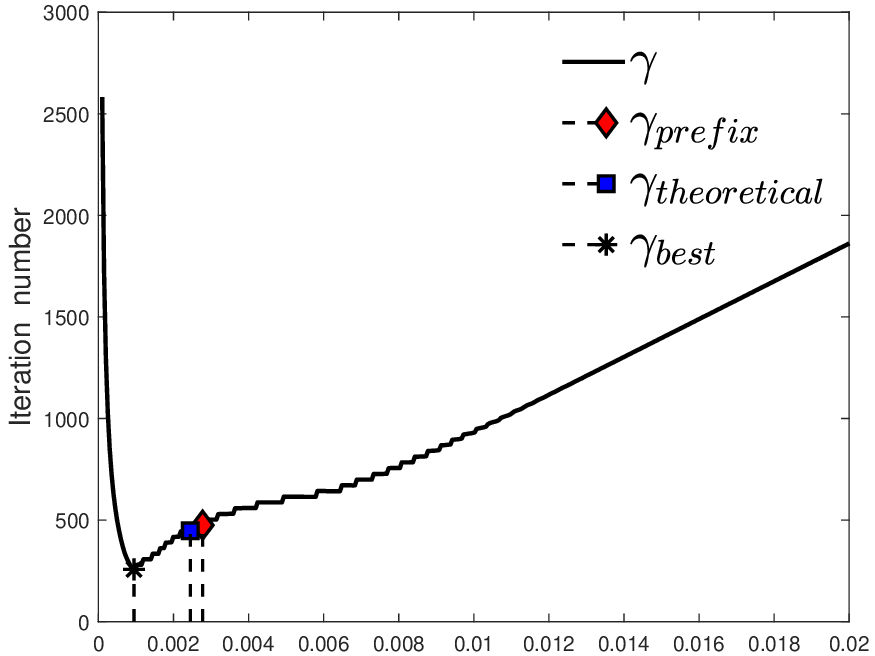} 	
				\caption{SR}
				\label{1b}
			\end{subfigure}
			\begin{subfigure}{0.3\textwidth}	
				\includegraphics[scale=0.25]{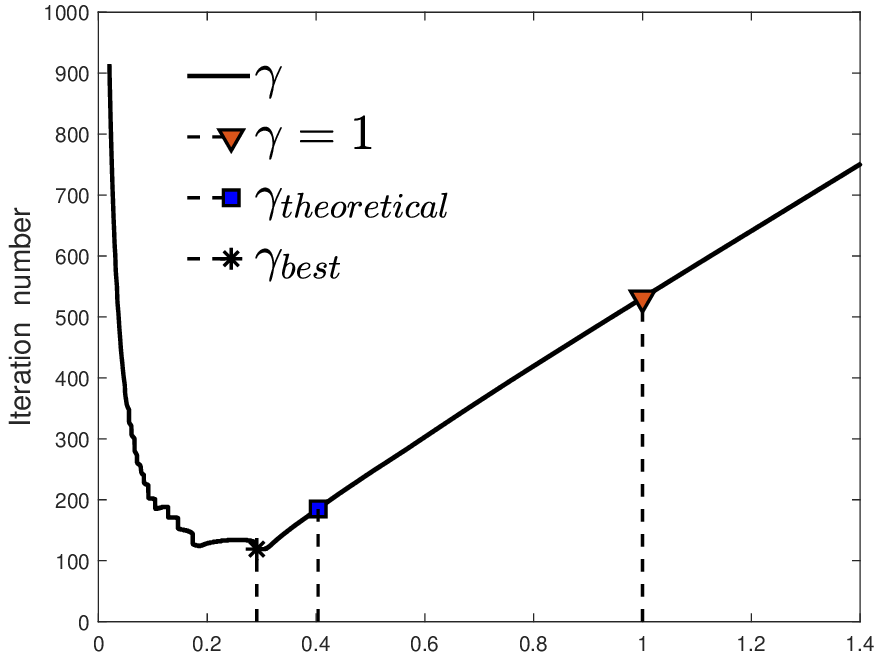} 	
				\caption{Lasso}
			\end{subfigure}
			\begin{subfigure}{0.3\textwidth}	
				\includegraphics[scale=0.25]{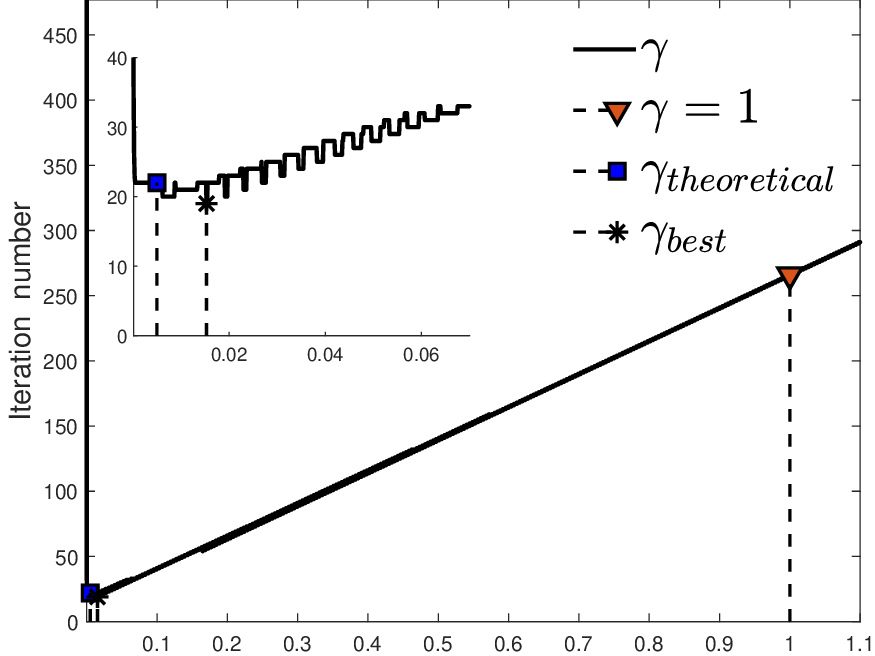} 	
				\caption{LAD}
			\end{subfigure}	
			\begin{subfigure}{0.3\textwidth}	
				\includegraphics[scale=0.25]{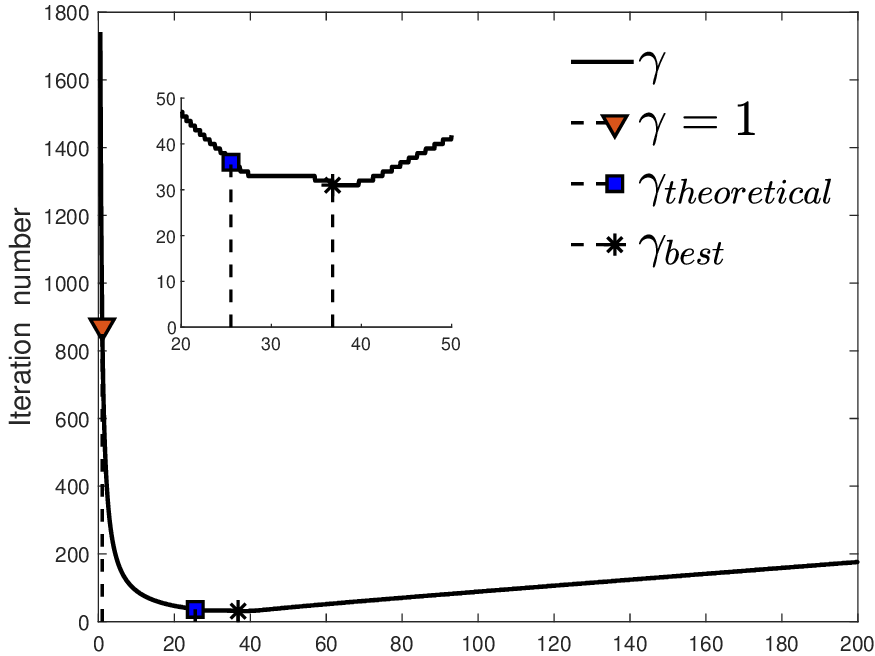} 	
				\caption{QP}
			\end{subfigure}
			\begin{subfigure}{0.3\textwidth}	
				\includegraphics[scale=0.25]{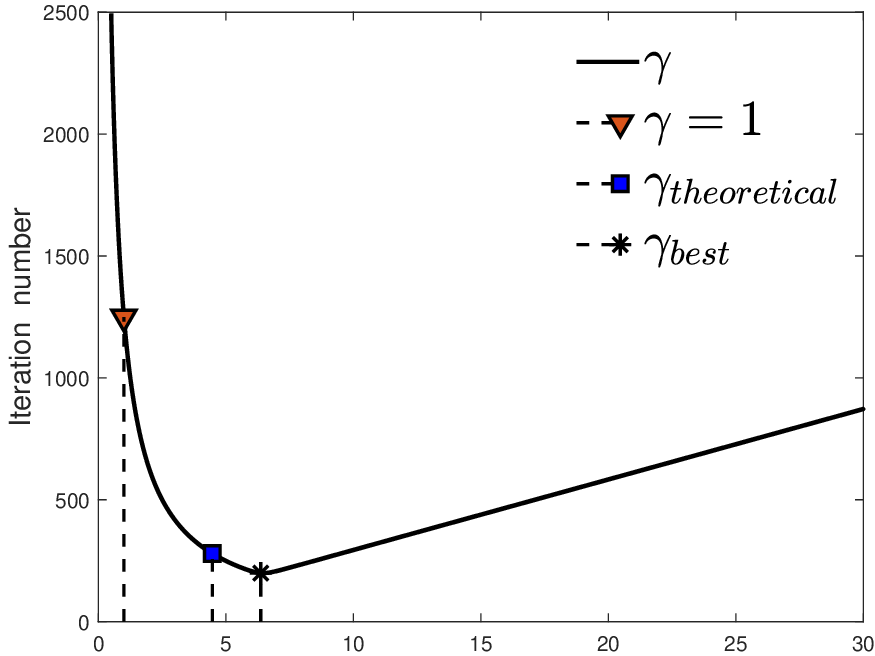} 	
				\caption{TV}
			\end{subfigure}
			\caption{Practical performance of different step-size choices of $\gamma$.}
		\end{center}
	\end{figure}
	\textbf{Data settings:} In all simulations, we use random Gaussian data $ \mathcal{N}(0, \sigma) $ and common  settings as in the literature. Particularly, we adapt the data settings from \cite{boyd12} for Lasso, LAD, QP, TV.  For applications QP, TV, all random data are from $ \mathcal{N}(0, 1) $.  For SR, due to the pre-fixed step-size is related to $\sigma$, we slightly change $ \sigma $ from 1 to 2 to verify the correctness of our formula. 
	For applications BQP, Lasso, LAD, the data appears well-conditioned if we use $ \mathcal{N}(0, 1) $. In this case,  $\gamma = 1$  appears sufficiently good and our results are less useful here.  We then promote an ill-conditioned structure by significantly changing the value of $\sigma$ of either $ \bm{A} $ or $ \bm{b} $, while keeping all other data the same as in \cite{boyd12}, which is mostly from $ \mathcal{N}(0, 1) $. Specifically, in BQP we generate the entries of $ \bm{A} $  from $ \mathcal{N}(0, 0.01) $; in Lasso, we promote an ill-conditioned structure of $ \bm{b} $ using $ \mathcal{N}(0, 50) $; in LAD, 5 randomly chosen entries of $\bm{b} \in \mathbb{R}^{500 \times 1} $ is increased from the original setting of $ \mathcal{N}(0, 100) $ in \cite{boyd12} to $ \mathcal{N}(0, 5000) $.
	
	%
	
	\textbf{Performance:} 
	For the pre-fixed step-sizes of BQP, SR as in  \eqref{prefix_bqp}, \eqref{prefix_sr}, we simply use the upper bound and the lower bound of $ d $, respectively. As can be seen from Figure \ref{1a}, \ref{1b}, the pre-fixed step-sizes performance of BQP and SR will be improved if we choose d slightly smaller than the upper bound and larger than the lower bound, respectively. Nevertheless, their performances are already outstanding and a more careful choice is out of our scope. One advantage of the pre-fixed step-size from direct estimation is that we have insights on how to further improve it, particularly if we have some additional knowledge of the data.

	The current iterates estimator $ ||\bm{\lambda}^k|| / ||\bm{x}^k||$  is not plotted in the above figure due to that it is an adaptive step-size. We find that its iteration number complexity performance is roughly the same as the optimal step-size (as the `blue square' above).

\bibliographystyle{plain}
\bibliography{Reference/ref1,Reference/Ref_S,Reference/ML_application,Reference/sr_applications}

	\appendix
	
	\section{Notation}
	Let $ \mathscr H $  denote a real Hilbert space  with an inner product $ \langle\cdot, \cdot\rangle $ equipped. The associated norm is given by $\Vert\cdot\Vert  $. 
	An operator $ \mathcal{T} $ on a Hilbert space is a point-to-set mapping $ \mathcal{T}:\mathscr H \rightarrow 2^\mathscr H $.
	We denote by  $ \Gamma_{0}(\mathscr{H}) $  the class of  closed, convex and proper (CCP) functions from $ \mathscr{H} $ to $ \mathbb{R} \cup \{\infty\} $, 
	by $ \circ $ the function composition,
	by $ \text{int}(\cdot) $ the  interior,
	by $ \text{sri}(\cdot) $ the strong relative interior,
	by $ \text{ran}(\cdot) $ the range,
	by $ \text{dom}(\cdot) $ the domain,
	by $ \mathfrak{B}(\mathscr{H})  \eqdef \{\mathcal{T}: \mathscr{H}\rightarrow\mathscr{H} \,\vert\, \mathcal{T}\,\, \text{is linear and continuous}\} $
	the space of bounded linear operators.
	The set operation $ C - D \eqdef \{x-y \,|\, x\in C, y\in D\} $.
	
	The uppercase bold, lowercase bold, and not bold letters are used for matrices, vectors, and scalars, respectively. The uppercase calligraphic letters, such as $ \mathcal{A} $, are used to denote  operators. Particularly, the identity operator is denoted by $ \mathcal{I} $. 
	
	\section{Preparatory materials}
	
		\subsection{Lipschitz continuous operators}
	
	\begin{defi}\label{def_Lipschitz}
		Let $ \mathcal{T}: \mathscr H \rightarrow \mathscr H$. Then,  $ \mathcal{T} $ is
		
		\vspace{4pt}
		\noindent(i) $L$-Lipschitz continuous if \,
		$ \Vert \mathcal{T}\bm{x} - \mathcal{T}\bm{y}  \Vert \leq L \Vert \bm{x} - \bm{y}  \Vert, \,\,\forall \bm{x},\bm{y} \in \mathscr H, \: L\geq 0 $;

		\vspace{4pt}
		\noindent(ii) a contraction if it is $L$-Lipschitz continuous with $L\in[0,1)$;   
		
		\vspace{4pt}
		\noindent(iii) nonexpansive if it is $1$-Lipschitz continuous, i.e., 
		$ \Vert \mathcal{T}\bm{x} - \mathcal{T}\bm{y}  \Vert \leq  \Vert \bm{x} - \bm{y}  \Vert, \forall \bm{x},\bm{y} \in \mathscr H $;
		
		\vspace{4pt}
		\noindent(iv) $1/L$-cocoercive if \,
		$ \frac{1}{L}	\Vert \mathcal{T}\bm{x} - \mathcal{T}\bm{y}  \Vert^2 \leq  \langle \mathcal{T}\bm{x} - \mathcal{T}\bm{y}, \bm{x} - \bm{y} \rangle, \,\,\forall \bm{x},\bm{y} \in \mathscr H,\: L> 0 $;
		
		\vspace{4pt}
		\noindent(v) firmly nonexpansive if it is $1$-cocoercive, i.e., 
		\begin{equation*}
			\Vert \mathcal{T}\bm{x} - \mathcal{T}\bm{y}  \Vert^2 \leq  \langle \mathcal{T}\bm{x} - \mathcal{T}\bm{y}, \bm{x} - \bm{y} \rangle, \quad\forall \bm{x},\bm{y} \in \mathscr H.
		\end{equation*}
	\end{defi}
	
	\begin{lem}[${1}/{2}$-averaging]\label{1/2 averaged operator}
		Let $ \mathcal{T} $ be nonexpansive. Then, the composition $ \frac{1}{2}\mathcal{I} + \frac{1}{2}\mathcal{T} $ is firmly nonexpansive. 
		Conversely, suppose $ 2\mathcal{A} - \mathcal{I} $ is nonexpansive. Then, $ \mathcal{A} $ is firmly nonexpansive.
	\end{lem}

	\subsection{Fixed-point characterization of ADMM}
	Here, we present the fixed-point characterization of ADMM. Proving the two types of fixed-point characterizations as in Section 3  (corresponding to the classical and new proximal operators) follows the same strategy. To avoid repeating, here we only prove the fixed-point characterization corresponding to our proposed new proximal operator.
	Let $  \bm{\psi}^{k+1} \eqdef \sqrt{\gamma}{\bm{x}}^{k+1} + {\bm{\lambda}}^k/\sqrt{\gamma} $. We will show that the sequence $ \{\bm{\psi}^{k+1}\} $ will converge to a fixed-point (if it exists).
	
	To start, we first rewrite the new ADMM iterates as in Section 3 using the definition $  \bm{\psi}^{k+1} \eqdef \sqrt{\gamma}{\bm{x}}^{k+1} + {\bm{\lambda}}^k/\sqrt{\gamma} $ into
		\begin{align}
		{\bm{x}}^{k+1} =&\,\,  \frac{1}{\sqrt{\gamma}}\textbf{Prox}_{f\frac{1}{\sqrt{\gamma}}}(2\sqrt{\gamma}{\bm{z}}^{k}- \bm{\psi}^{k}),\nonumber\\				
		{\bm{z}}^{k+1} =&\,\,  \frac{1}{\sqrt{\gamma}}\textbf{Prox}_{g\frac{1}{\sqrt{\gamma}}}(\bm{\psi}^{k}),\nonumber\\
		{\bm{\lambda}}^{k+1} =&\,\, \sqrt{\gamma} (\bm{\psi}^{k+1} - \sqrt{\gamma}{\bm{z}}^{k+1}).
		\end{align}
	Invoking the above iterates definitions, we arrive at 
	\begin{align}
		 \bm{\psi}^{k+1} 
		 =& \sqrt{\gamma}{\bm{x}}^{k+1} + {\bm{\lambda}}^k/\sqrt{\gamma}, \nonumber\\
		 =& \textbf{Prox}_{f\frac{1}{\sqrt{\gamma}}}(2\sqrt{\gamma}{\bm{z}}^{k}- \bm{\psi}^{k}) + \bm{\psi}^{k} - \sqrt{\gamma}{\bm{z}}^{k}, \nonumber\\
		 =& \textbf{Prox}_{f\frac{1}{\sqrt{\gamma}}}(2\textbf{Prox}_{g\frac{1}{\sqrt{\gamma}}}(\bm{\psi}^{k})- \bm{\psi}^{k}) + \bm{\psi}^{k} - \textbf{Prox}_{g\frac{1}{\sqrt{\gamma}}}(\bm{\psi}^{k}), \nonumber\\
		 =& \textbf{Prox}_{f\frac{1}{\sqrt{\gamma}}}\circ(2\textbf{Prox}_{g\frac{1}{\sqrt{\gamma}}} - \mathcal{I})(\bm{\psi}^{k}) - 
		  \frac{1}{2}(2\textbf{Prox}_{g\frac{1}{\sqrt{\gamma}}}   (\bm{\psi}^{k}) -\bm{\psi}^{k}) + \frac{1}{2}\bm{\psi}^{k}, \nonumber\\
		 =& \frac{1}{2}(2\textbf{Prox}_{f\frac{1}{\sqrt{\gamma}}} - \mathcal{I})\circ(2\textbf{Prox}_{g\frac{1}{\sqrt{\gamma}}} - \mathcal{I})(\bm{\psi}^{k}) + \frac{1}{2}\bm{\psi}^{k}, \nonumber\\
		 =&  \mathcal{F}\bm{\psi}^{k},
	\end{align}
	where $ \mathcal{F} \eqdef \frac{1}{2}(2\textbf{Prox}_{f\frac{1}{\sqrt{\gamma}}} - \mathcal{I})\circ(2\textbf{Prox}_{g\frac{1}{\sqrt{\gamma}}} - \mathcal{I}) + \frac{1}{2}\mathcal{I}$, and where $ \mathcal{I} $ denotes the identity operator.
	
	By Proposition 2.1, the new proximal operator is firmly nonexpansive. Then, by Lemma \ref{1/2 averaged operator}, the term $ 2\textbf{Prox}_{f\frac{1}{\sqrt{\gamma}}} - \mathcal{I} $ is nonexpansive. The composition of two such nonexpansive operators is also nonexpansive, i.e., $ (2\textbf{Prox}_{f\frac{1}{\sqrt{\gamma}}} - \mathcal{I})\circ(2\textbf{Prox}_{g\frac{1}{\sqrt{\gamma}}} - \mathcal{I}) $ is nonexpansive. Invoking Lemma \ref{1/2 averaged operator} again, due to $ \mathcal{F} $ admits a 1/2-averaged form, it is  firmly nonexpansive.
	Due to the  firm nonexpansiveness, we can conclude that the sequence $ \{\bm{\psi}^{k+1}\} $ will converge to a fixed-point (if it exists).

	\subsection{SDP properties}
	\begin{lem}
		Given an indicator function $ \delta_{\mathbb{S}_+}  $, its associated proximal operator is reduced to a projection operator
		\begin{equation}
			\textbf{Prox}_{\delta_{\mathbb{S}_+^N} }(\cdot) = \Pi_{\mathbb{S}^N_+}(\cdot). 
		\end{equation}
	\end{lem}
	\begin{lem}\label{lem_sdp2}
		Let $ (\bm{X}^\star, \bm{\Lambda}^\star) $ be the primal-dual solution pair for solving SDP as in (5.1) via ADMM. The following holds:
		\begin{equation}\label{property_2}
			\bm{X}^\star \succeq 0,\quad \bm{\Lambda}^\star \preceq 0,\quad  \langle \bm{X}^\star, \bm{\Lambda}^\star \rangle = 0.
		\end{equation}	
	\end{lem}
	\begin{proof}
		By the update rule of ADMM, we have
		\begin{equation}
			\bm{Z}^\star =\,\, \textbf{Prox}_{g}(\bm{X}^\star + \bm{\Lambda}^\star) = \Pi_{\mathbb{S}^N_+}(\bm{X}^\star + \bm{\Lambda}^\star) \succeq 0.
		\end{equation}
		Since $ \bm{X}^\star = \bm{Z}^\star $, we have $ \bm{X}^\star \succeq 0 $.

		For the dual solution $ \bm{\Lambda}^\star  $, we have
		\begin{align}
			\bm{\Lambda}^\star 
			&= \bm{X}^\star + \bm{\Lambda}^\star - \bm{X}^\star \nonumber\\
			&= \Pi_{\mathbb{S}^N_+}(\bm{X}^\star + \bm{\Lambda}^\star) + \Pi_{\mathbb{S}^N_-}(\bm{X}^\star + \bm{\Lambda}^\star) - \bm{X}^\star \nonumber\\
			&= \Pi_{\mathbb{S}^N_-}(\bm{X}^\star + \bm{\Lambda}^\star) \preceq 0
		\end{align}
		
		At last,
		\begin{equation}
			\langle \bm{X}^\star, \bm{\Lambda}^\star \rangle = \langle \Pi_{\mathbb{S}^N_+}(\bm{X}^\star + \bm{\Lambda}^\star), \Pi_{\mathbb{S}^N_-}(\bm{X}^\star + \bm{\Lambda}^\star) \rangle = 0.
		\end{equation}
		The proof is now concluded.
	\end{proof}

	\section{Proofs of main results}
	\subsection{proof of Proposition 2.1, 2.2 }
	Here, we directly prove Proposition 2.2, which is a generalised version of Proposition 2.1. We will appeal to the maximal monotone operator theory.
	
	To start, we employ the following  lemmas:
	
	\begin{lem}\cite[Proposition 6.19]{bauschke2017convex}\label{lem1}
		Let $ C, D $ be convex subsets of $ \mathscr H $, let $  \mathcal{L} \in \mathfrak{B}(\mathscr{H}) $. Suppose $ D \bigcap \text{int}\,\, \mathcal{L}(C) \neq \emptyset$ or $ \mathcal{L}(C) \bigcap \text{int}\,\, D \neq \emptyset$. Then, $ 0\in \text{sri}\, (D - \mathcal{L}(C) ) $.
	\end{lem}
	
	\vspace{1pt}
	
	\begin{lem}\cite[Corollary 16.53]{bauschke2017convex}\label{dom_ran}
		Let $ f \in \Gamma_{0}(\mathscr{H}) $ and $  \mathcal{L} \in \mathfrak{B}(\mathscr{H}) $. Suppose $ 0\in \text{sri}\, (\text{dom}(f) -\text{ran}(\mathcal{L}) ) $. Then, 
		\begin{equation*}
			\partial (f\circ\mathcal{L}) = \mathcal{L}^*\circ\partial f\circ\mathcal{L}.
		\end{equation*}
	\end{lem}

	\begin{lem}\cite[Proposition 23.25]{bauschke2017convex}\label{coro_maxm}
		Let $\mathcal{A}: \mathscr{H} \rightarrow2^\mathscr{H}$ be maximal monotone. Suppose  $ \mathcal{L} \in \mathfrak{B}(\mathscr{H}) $ is such that $ \mathcal{L}\mathcal{L}^* $ is invertible. Let $ \mathcal{B} \eqdef \mathcal{L}^*\mathcal{A}\mathcal{L} $. Then, 	
		$\mathcal{B}: \mathscr{H}\rightarrow2^\mathscr{H}$ is maximal monotone.	
	\end{lem}

	By definition, we directly have the following resolvent characterization:
	\begin{equation}\label{resolvent1}
		\textbf{Prox}_{f\mathcal{S}} = (\mathcal{I} + \partial (f\circ\mathcal{S}))^{-1}.
	\end{equation}
	The key is to prove $ \partial (f\circ\mathcal{S}) $ being maximal monotone. 
	
	We first apply Lemma \ref{lem1}. By our definition $ f $ is CCP and  $ \mathcal{S} \in \mathfrak{B}(\mathscr{H})$, we can substitute the set $ D $ there with  $ \text{dom}(f) $ and substitute $ \mathcal{L}(C) $ with $ \text{ran}(\mathcal{S})  $. 
	Clearly, $ \text{ran}(\mathcal{S}) \bigcap \text{int}\,\, \text{dom}(f) \neq \emptyset$.
	Therefore, we have  $ 0\in \text{sri}\, (\text{dom}(f) -\text{ran}(\mathcal{L}) ) $. 
	
	Then, by  Lemma \ref{dom_ran},  we can rewrite \eqref{resolvent1} into
	\begin{equation}
		\textbf{Prox}_{f\mathcal{S}}  =  (\mathcal{I} + \mathcal{S}^*\circ\partial f\circ\mathcal{S})^{-1}.
	\end{equation}
	Since  $\mathcal{S}$ is bijective, we have  that $ \mathcal{S}\mathcal{S}^* $ is invertible. In view of Lemma \ref{coro_maxm}, the composite operator $ \mathcal{S}^*\circ\partial f\circ\mathcal{S} $ is therefore maximal monotone. Hence, the operator $ (\mathcal{I} + \mathcal{S}^*\circ\partial f\circ\mathcal{S})^{-1} $ is firmly nonexpansive and single-valued. The proof is now concluded.

	\subsection{proof of Definition 2.2, 2.4}
	We directly prove the generalized version as in Definition 2.4.
	
	Given $ \textbf{Prox}_{f\mathcal{S}} $ with $ \mathcal{S} $ being  bijective, linear and continuous, we aim to show
	\begin{equation}
		\underset{\bm{z}}{\text{argmin}} \,\, f(\mathcal{S}\bm{z})+ \frac{1}{2}\Vert \bm{z} - \cdot\Vert^2
		=  \mathcal{S}^{-1} \underset{\bm{z}}{\text{argmin}} \,\, f(\bm{z})+ \frac{1}{2}\Vert \mathcal{S}^{-1}\bm{z} - \cdot\Vert^2.
	\end{equation}
	Let $ \tilde{z} \eqdef \mathcal{S}\bm{z}$, and let $ {z}^\star $ be the minimizer to $ f(\mathcal{S}\bm{z})+ \frac{1}{2}\Vert \bm{z} - \cdot\Vert^2 $. Then, 
	$ \tilde{z}^\star = \mathcal{S}\bm{z}^\star $ would be the minimizer to $ f(\tilde{z})+ \frac{1}{2}\Vert \mathcal{S}^{-1}\tilde{z} - \cdot\Vert^2 $. 
	Hence, $ \bm{z}^\star = \mathcal{S}^{-1} \underset{\tilde{\bm{z}}}{\text{argmin}} \,\,f(\tilde{z})+ \frac{1}{2}\Vert \mathcal{S}^{-1}\tilde{z} - \cdot\Vert^2 $, which concludes the proof.

	\subsection{proof of Lemma 2.1}
	To show the translation rule  in Lemma 2.1, we first recall the following definitions:
	\begin{align}
		\textbf{Prox}_{\gamma f}(\cdot) =& \,\underset{\bm{z}}{\text{argmin}} \,\,   f(\bm{z}) + \frac{1}{2\gamma}\Vert \bm{z} - \cdot\Vert^2,\nonumber\\		
		\textbf{Prox}_{f \rho} (\cdot) =& \frac{1}{\rho}\,\underset{\bm{z}}{\text{argmin}} \,\, f(\bm{z})+ \frac{1}{2}\Vert \frac{\bm{z} }{\rho}- \cdot\Vert^2.
	\end{align}
	
	For the first relation in Lemma 2.1, we have 
	\begin{align}
		\frac{1}{\rho}\textbf{Prox}_{\rho^2f} ( \rho{\bm{v}} ) 
		=& \frac{1}{\rho}\underset{\bm{z}}{\text{argmin}} \,   f(\bm{z}) + \frac{1}{2\rho^2}\Vert \bm{z} - \rho\bm{v}\Vert^2 \nonumber\\
		=&  \frac{1}{\rho}\underset{\bm{z}}{\text{argmin}} \,   f(\bm{z}) + \frac{1}{2}\Vert \frac{\bm{z}}{\rho} - \bm{v}\Vert^2 \nonumber\\
		=& \textbf{Prox}_{f \rho}(\bm{v}).
	\end{align}
	
	For the second relation  in Lemma 2.1,  we have 
	\begin{align}
		\sqrt{\gamma}\textbf{Prox}_{f\sqrt{\gamma}} ( \frac{\bm{v}}{\sqrt{\gamma}}) 
		=& \sqrt{\gamma} \,\frac{1}{\sqrt{\gamma}}\underset{\bm{z}}{\text{argmin}}  f(\bm{z}) + \frac{1}{2}\Vert \frac{\bm{z}}{\sqrt{\gamma}} - \frac{\bm{v}}{\sqrt{\gamma}}\Vert^2\nonumber\\
		=& \underset{\bm{z}}{\text{argmin}}   f(\bm{z}) + \frac{1}{2\gamma}\Vert \bm{z} - \bm{v}\Vert^2\nonumber\\
		=&\textbf{Prox}_{\gamma f}(\bm{v}).
	\end{align}
	The proof is now concluded.

	\subsection{proof of Lemma 2.2}
		To show the translation rule  in Lemma 2.2, we first recall the following definitions:
	\begin{align}
		\textbf{Prox}_{\mathcal{M} f}(\cdot) =& \,\underset{\bm{z}}{\text{argmin}} \,\,   f(\bm{z}) + \frac{1}{2}\Vert \bm{z} - \cdot\Vert^2_{\mathcal{M}^{-1}},\nonumber\\
		\textbf{Prox}_{f \mathcal{S}} (\cdot) =&\mathcal{S}^{-1}\underset{\bm{z}}{\text{argmin}} \,\, f(\bm{z})+ \frac{1}{2}\Vert \mathcal{S}^{-1}\bm{z}- \cdot\Vert^2.
	\end{align}
	
	The key here is that for any positive definite operator $ \mathcal{M} \succ 0$, it admits a decomposition $ \mathcal{M} = \mathcal{S}\mathcal{S}^* $, where $ \mathcal{S}  $ is a  bijective, linear and continuous operator and where $ \mathcal{S}^* $ denotes the adjoint operator.
	
	For the first relation in Lemma 2.2, we have 
	\begin{align}
		\mathcal{S}^{-1} \textbf{Prox}_{\mathcal{M} f}(\mathcal{S}\bm{v}) 
		=& \mathcal{S}^{-1}\underset{\bm{z}}{\text{argmin}} \,\,  f(\bm{z}) + \frac{1}{2}\Vert \bm{z} -\mathcal{S}\bm{v}\Vert^2_{\mathcal{M}^{-1}}\nonumber\\
		=& \mathcal{S}^{-1}\underset{\bm{z}}{\text{argmin}} \,\,  f(\bm{z}) + \frac{1}{2}  \langle\bm{z} -\mathcal{S}\bm{v},  (\mathcal{S}\mathcal{S}^{*})^{-1} (\bm{z} -\mathcal{S}\bm{v}) \rangle\nonumber\\
		=& \mathcal{S}^{-1}\underset{\bm{z}}{\text{argmin}} \,\,  f(\bm{z}) + \frac{1}{2}\Vert \mathcal{S}^{-1}\bm{z} - \mathcal{S}^{-1}\mathcal{S}\bm{v}\Vert^2\nonumber\\
		=&  \textbf{Prox}_{f\mathcal{S} }(\bm{v}).
	\end{align}
	
	For the second relation in Lemma 2.2, we have 
	\begin{align}
		  \mathcal{S}\textbf{Prox}_{f\mathcal{S}}( \mathcal{S}^{-1}\bm{v} ) 
		=& \mathcal{S}\mathcal{S}^{-1}\,\underset{\bm{z}}{\text{argmin}} \,\,f(\bm{z}) + \frac{1}{2}\Vert \mathcal{S}^{-1}\bm{z} - \mathcal{S}^{-1}\bm{v}\Vert^2\nonumber\\
		=& \underset{\bm{z}}{\text{argmin}} \,\,f(\bm{z}) + \frac{1}{2}  \langle\bm{z} -\mathcal{S}\bm{v},  (\mathcal{S}\mathcal{S}^{*})^{-1} (\bm{z} -\bm{v}) \rangle\nonumber\\
		=& \underset{\bm{z}}{\text{argmin}} \,\, f(\bm{z}) + \frac{1}{2}\Vert \bm{z} -\bm{v}\Vert^2_{(\mathcal{S}\mathcal{S}^{*})^{-1}}\nonumber\\
		= & \textbf{Prox}_{\mathcal{M} f}(\bm{v}).
	\end{align}
	The proof is now concluded.

	\subsection{proof of Proposition 2.3}
	Our proof is  inspired by \cite{cai2010singular} where the authors exploit the sub-gradient structure of the nuclear norm. 
	We therefore  study  the sub-gradient structure of the $ l_1 $ norm. 
	Recall that we aim to prove 
	\begin{equation}
		\bm{P}^\text{T}\mathcal{T} ((\bm{P}^{-1})^\text{T}\bm{v}) = \textbf{Prox}_{f\bm{P}}(\bm{v}),
	\end{equation}
	where $\bm{P}$ is a full-rank square matrix with orthogonal columns.

	To start, we first rewrite the above goal into proving
	\begin{equation}\label{l1_diag}
		(\bm{D}^\text{T}\bm{D})^{-1}\mathcal{T} (\bm{D}^\text{T}\bm{v}) 
		=  \underset{\bm{x}}{\text{argmin}} \,\, \Vert \bm{x} \Vert_1 + \frac{1}{2}\Vert\bm{D}\bm{x} - \bm{v} \Vert^2  ,
	\end{equation}
	where $ \bm{D} \eqdef \bm{P}^{-1} $, and the alternative definition $ \textbf{Prox}_{f \bm{P}} (\cdot) =\bm{P}^{-1}\underset{\bm{z}}{\text{argmin}} \,\, f(\bm{z})+ \frac{1}{2}\Vert \bm{P}^{-1}\bm{z}- \cdot\Vert^2 $ is invoked (recall from Definition 2.4).
	
	Let $ \bm{x}^\star $ be the minimizer to the above problem. Then,
	\begin{equation}
		\bm{D}^\text{T}\bm{v} - \bm{D}^\text{T}\bm{D}\bm{x}^\star \in \partial\Vert\bm{x}^\star \Vert_1.
	\end{equation}
	Suppose $ (\bm{D}^\text{T}\bm{D})^{-1}\mathcal{T} (\bm{D}^\text{T}\bm{v}) $  as in \eqref{l1_diag} is  correct, then we can substitute $ \bm{x}^\star $ from above with it. Our goal therefore becomes  proving the following sub-gradient relation:
	\begin{equation}\label{1_norm}
		\bm{D}^\text{T}\bm{v} - \mathcal{T} (\bm{D}^\text{T}\bm{v}) \in \partial \Vert (\bm{D}^\text{T}\bm{D})^{-1} \mathcal{T} (\bm{D}^\text{T}\bm{v})\Vert_1.
	\end{equation}

	Below, we will prove \eqref{1_norm}.
	
	 First, define the following decomposition:
	\begin{equation}
		\bm{D}^\text{T}\bm{v} \eqdef  \bm{y}_0 + \bm{y}_1,
	\end{equation}	
	where $ \bm{y}_0  $ (resp., $ \bm{y}_1 $) has the absolute value of all its non-zero entries  larger than 1 (resp.,  less or equal to 1). Such a decomposition implies the relation $ \bm{y}_0 \odot \bm{y}_1 = \bm{0}$, where $ \odot $ denotes  Hadamard product, i.e., element-wise multiplication.	
	
	Applying the soft thresholding operator $ \mathcal{T} $  to above yields
	\begin{equation}
		\mathcal{T} (\bm{D}^\text{T}\bm{v}) =  \bm{y}_0  - \text{sgn}(\bm{y}_0).
	\end{equation}
	Then, the left-hand side of \eqref{1_norm} can be written as
	\begin{equation}\label{rec_1}
		\bm{D}^\text{T}\bm{v} - \mathcal{T} (\bm{D}^\text{T}\bm{v}) = \text{sgn}(\bm{y}_0) + \bm{y}_1.
	\end{equation}
	
	Now consider the right-hand side of \eqref{1_norm}. 
	Similar to the nuclear norm characterization as in \cite{cai2008singular}, the subdifferential characterization of the $ l_1 $ norm can be given by
	\begin{equation}
		\partial \Vert \bm{x} \Vert_1 = \{ \text{sgn}(\bm{x}) + \bm{w} \,\vert\,	\bm{x}\odot \bm{w} = 0, \,\Vert \bm{w} \Vert_\infty \leq 1 \}.
	\end{equation}		
	Since by definition $ \Vert\bm{y}_1\Vert_\infty \leq 1$ and $ \bm{y}_0 \odot \bm{y}_1 = \bm{0}$, we have
	\begin{equation}
		\text{sgn}(\bm{y}_0) + \bm{y}_1  \in \partial \Vert \bm{y}_0 \Vert_1.
	\end{equation}
	Substituting \eqref{rec_1} to the above yields
	\begin{equation}
		\bm{D}^\text{T}\bm{v} - \mathcal{T} (\bm{D}^\text{T}\bm{v}) \in \partial \Vert \bm{y}_0 \Vert_1.
	\end{equation}
	Comparing the above to the final goal \eqref{1_norm}, all what left is to show 
	\begin{equation}
		\partial \Vert \bm{y}_0 \Vert_1 = \partial \Vert (\bm{D}^\text{T}\bm{D})^{-1} \mathcal{T} (\bm{D}^\text{T}\bm{v})\Vert_1.
	\end{equation}
	This is equivalent to showing 
	\begin{equation}\label{to_s1}
		\text{sgn}(\bm{y}_0) = \text{sgn}((\bm{D}^\text{T}\bm{D})^{-1} \mathcal{T} (\bm{D}^\text{T}\bm{v})).
	\end{equation}
	Since by assumption $ \bm{D} $ has orthogonal columns, we have $ \bm{D}^\text{T}\bm{D} $ being a diagonal matrix and so its inverse.
	Hence, we can define  
	\begin{equation}
		(\bm{D}^\text{T}\bm{D})^{-1} \eqdef \text{Diag}(\bm{z}),
	\end{equation}
	where $ \bm{z}  $ is a positive vector.
	The right-hand side  of \eqref{to_s1} now can be rewritten into
	\begin{align}
		\text{sgn}(\,\text{Diag}(\bm{z}) (\bm{y}_0  - \text{sgn}(\bm{y}_0))\,) 
		&= \text{sgn}(\,\bm{z}\odot \text{sgn}(\bm{y}_0) \odot (|\bm{y}_0| - \bm{1})\,) \\
		&= \text{sgn}(\bm{y}_0) \odot \text{sgn} (\bm{z}\odot (|\bm{y}_0| - \bm{1})). \nonumber
	\end{align}
	Since by definition $ (|\bm{y}_0| - \bm{1}) $  and $ \bm{z} $ are positive vectors, their Hadamard product  therefore yields  a positive vector.
	That said,
	\begin{equation}
		\text{sgn}(\bm{y}_0) \odot \text{sgn} (\bm{z}\odot (|\bm{y}_0| - \bm{1})) = \text{sgn}(\bm{y}_0) \odot \bm{1} = \text{sgn}(\bm{y}_0), 
	\end{equation}	 
	which proves \eqref{to_s1}. The proof is therefore concluded.

	\subsection{proof of extended Moreau decompositions }	
Here, we first prove eq. (2.10), which is the extended Moreau decomposition in terms of the new proximal  operator. Then, we prove the classical one as in eq. (2.9).
We will need the following lemma:
	
	\begin{lem}\cite[Proposition 13.23]{bauschke2017convex}\label{lem_moreau}
		Let $ f : \mathscr H \rightarrow  \, ]-\infty, +\infty]$. Then the following holds:
		
		(i) (for $\alpha \in \mathbb{R}_{++}$ ) $ (\alpha f)^* = \alpha f^* (\cdot/\alpha) $. 
		
		\vspace{1pt}
		(ii) 	Let $ \mathcal{L}\in \mathfrak{B}(\mathscr{H}) $ be bijective. Then $ (f\circ\mathcal{L})^* = f^*\circ {\mathcal{L}^*}^{-1}. $
	\end{lem}
	
	$\bullet$ For generality, we directly prove the operator parameter case, and the reduction to the scalar case as in eq. (2.10)  is straightforward.
	
	First, recall the basic Moreau decomposition 
	\begin{equation}
		\bm{v} =\,\textbf{Prox}_{f} (\bm{v}) +  \textbf{Prox}_{f^*} (\bm{v})
	\end{equation}
	
	Given a bijective parameter $\mathcal{S} \in \mathfrak{B}(\mathscr{H})$,
	substitute $ f $ above with $ f\circ\mathcal{S} $, we arrive at		
	\begin{equation}
		\bm{v} =\,\textbf{Prox}_{f\circ\mathcal{S}} (\bm{v}) +  \textbf{Prox}_{(f\circ\mathcal{S})^*} (\bm{v}) 
		=\,\textbf{Prox}_{f\circ\mathcal{S}} (\bm{v}) +  \textbf{Prox}_{f^*\circ{\mathcal{S}^*}^{-1}} (\bm{v}), 
	\end{equation}
	where the second equality is by invoking Lemma \ref{lem_moreau} (ii).
	
	$\bullet$ Now we prove the classical Moreau decomposition extension. 
	Consider a scalar parameter $\gamma > 0$, and substitute $ f $ above with $ \gamma f $, we arrive at
	\begin{equation}\label{eq_mor}
		\bm{v} =\,\textbf{Prox}_{\gamma f} (\bm{v}) +  \textbf{Prox}_{(\gamma f)^*} (\bm{v}) 
		= \,\textbf{Prox}_{\gamma f} (\bm{v}) +  \textbf{Prox}_{\gamma f^* \frac{1}{\gamma}} (\bm{v}), 
	\end{equation}
	where the second equality is by invoking Lemma \ref{lem_moreau} (i).
	Recall the translation rule   $ \textbf{Prox}_{f \frac{1}{\gamma}}(\bm{v}) = {\gamma}\textbf{Prox}_{\frac{1}{\gamma^2}f} (  \frac{\bm{v}}{\gamma}) $ in Lemma 2.1. We obtain
	\begin{equation}
		\textbf{Prox}_{\gamma f^* \frac{1}{\gamma}} (\bm{v})	= \gamma\textbf{Prox}_{\frac{1}{\gamma} f^*} ( \frac{\bm{v}}{\gamma}).
	\end{equation}
	Substituting the above to \eqref{eq_mor}, we arrive at 
	\begin{equation}
		\bm{v} =\,\textbf{Prox}_{\gamma f} (\bm{v}) +  \gamma\textbf{Prox}_{\frac{1}{\gamma} f^*} ( \frac{\bm{v}}{\gamma}). 
	\end{equation}
	The proof is now concluded.

	\subsection{proof of Proposition 4.1 (convergence rate)}
	
	We will need the following lemma:
	\begin{lem}[strictly decreasing error]\label{lem_rate}
		Suppose an operator $ \mathcal{F} $  is $1/L$-cocoercive with  $ L \in (0,1) $.
		For a set of fixed-point iterates $ \{\bm{\psi}^k\} $  that converges to  $\bm{\psi}^\star\in \text{Fix} (\mathcal{F}) $, the following holds:
		\begin{equation}
			\frac{2-L}{L}\Vert \bm{\psi}^{k+1} - \bm{\psi}^\star \Vert^2 - \Vert \bm{\psi}^{k} - \bm{\psi}^\star \Vert^2 \leq -\Vert \bm{\psi}^{k+1} - \bm{\psi}^k  \Vert^2.
		\end{equation}
	\end{lem}
	\begin{proof}
		By Definition \ref{def_Lipschitz} (iv), one has
		\begin{align*}
			& \Vert \mathcal{F}\bm{\psi}^{k} - \mathcal{F}\bm{\psi}^\star  \Vert^2 \leq L \langle  \mathcal{F}\bm{\psi}^{k} - \mathcal{F}\bm{\psi}^\star , 
			\bm{\psi}^{k} -\bm{\psi}^\star\rangle \\
			\iff& \Vert \bm{\psi}^{k+1} - \bm{\psi}^\star \Vert^2 \leq L\langle  \bm{\psi}^{k+1} - \bm{\psi}^\star, \bm{\psi}^{k} -\bm{\psi}^\star\rangle \\
			\iff& 0\leq L\langle  \bm{\psi}^{k+1} - \bm{\psi}^\star,  \bm{\psi}^{k} -\bm{\psi}^{k+1} \rangle + (L-1) \Vert \bm{\psi}^{k+1} - \bm{\psi}^\star \Vert^2 \\
			\iff& 0\leq L\Vert \bm{\psi}^{k}  - \bm{\psi}^\star \Vert^2 -L\Vert \bm{\psi}^{k+1} - \bm{\psi}^\star \Vert^2 - L\Vert \bm{\psi}^{k+1} - \bm{\psi}^k\Vert^2  + 2(L-1) \Vert \bm{\psi}^{k+1} - \bm{\psi}^\star \Vert^2   \label{Pythagoras} \\		
			\iff& \frac{2-L}{L}\Vert \bm{\psi}^{k+1} - \bm{\psi}^\star \Vert^2 - \Vert \bm{\psi}^{k} - \bm{\psi}^\star \Vert^2 \leq -\Vert \bm{\psi}^{k+1} - \bm{\psi}^k  \Vert^2,
		\end{align*}
		where the  Pythagoras relation is invoked, i.e., 
		\begin{equation*}
			2\langle a-b, c-a \rangle = \Vert c - b \Vert^2 -  \Vert a - b \Vert^2  - \Vert c - a \Vert^2
		\end{equation*}
	\end{proof}

	$\bullet$ Now we are ready to prove Proposition 4.1, we would start with the Lipschitz constant  $ L \in (0,1) $ case, which states
	\begin{equation}\label{eq_rate1}
		\Vert \bm{\psi}^{k+1} - \bm{\psi}^k  \Vert^2 \leq (\frac{L}{2-L})^k \cdot\frac{ 1 - \frac{L}{2-L}}{1 -(\frac{L}{2-L})^{k+1} }\Vert \bm{\psi}^\star - \bm{\psi}^0  \Vert^2. 
	\end{equation}

	For light of notation, let $ a \eqdef \frac{L}{2-L} $.
	Lemma \ref{lem_rate} can be rewritten into
	\begin{equation}
		\Vert \bm{\psi}^{t+1} - \bm{\psi}^t \Vert^2  \leq \Vert \bm{\psi}^{t} - \bm{\psi}^\star \Vert^2 - \frac{1}{a}\Vert \bm{\psi}^{t+1} - \bm{\psi}^\star \Vert^2 	
	\end{equation}
	Therefore
	\begin{equation}\label{rate_1}
		\sum_{t=0}^{k} a^{-t} \Vert \bm{\psi}^{t+1} - \bm{\psi}^t  \Vert^2 \leq \Vert \bm{\psi}^\star - \bm{\psi}^0  \Vert^2.
	\end{equation}
	By Definition \ref{def_Lipschitz} (i), one has
	$ \Vert \bm{\psi}^{k+1} - \bm{\psi}^k  \Vert^2 \leq \Vert \bm{\psi}^{k} - \bm{\psi}^{k-1} \Vert^2 $.
	It follows that
	\begin{align}
		&	\sum_{t=0}^{k} a^{-t} \Vert \bm{\psi}^{k+1} - \bm{\psi}^k  \Vert^2 \leq \sum_{t=0}^{k} a^{-t} \Vert \bm{\psi}^{t+1} - \bm{\psi}^t\Vert^2 \label{eq_L1}\\
		\iff&	\frac{1-a^{-(k+1)}}{1-a^{-1}} \Vert \bm{\psi}^{k+1} - \bm{\psi}^k\Vert^2   \leq \sum_{t=0}^{k} a^{-t} \Vert \bm{\psi}^{t+1} - \bm{\psi}^t\Vert^2 \label{rate_2},
	\end{align}
	where we have used the geometric sum formula in the last line and let us note that  the quantity $  \Vert \bm{\psi}^{k+1} - \bm{\psi}^k  \Vert^2 $ is a constant under $ \sum_{t} $.
	
	Adding (\ref{rate_1}) and (\ref{rate_2}) yields the convergence rate bound
	\begin{equation}
		\Vert \bm{\psi}^{k+1} - \bm{\psi}^k\Vert^2   \leq    a^{k}\cdot\frac{1-a} {1-a^{k+1}}  \Vert \bm{\psi}^\star - \bm{\psi}^0  \Vert^2. 
	\end{equation}	
	Invoking $ a \eqdef \frac{L}{2-L} $ gives \eqref{eq_rate1}.

	\vspace{12pt}

	$\bullet$ The no strong assumption case corresponds to $ L = 1 $, which states
	\begin{equation}\label{rate_basic}
		\Vert \bm{\psi}^{k+1} - \bm{\psi}^k  \Vert^2 \leq \frac{1}{k+1} \Vert \bm{\psi}^\star - \bm{\psi}^0  \Vert^2.
	\end{equation}
	The proof follows the same strategy except we have $ a = 1 $ now. To avoid repeating, we can directly start from \eqref{rate_1} and \eqref{eq_L1}, which now gives
	\begin{align}
		\sum_{t=0}^{k} \Vert \bm{\psi}^{t+1} - \bm{\psi}^t  \Vert^2  &\leq \Vert \bm{\psi}^\star - \bm{\psi}^0  \Vert^2\\
		\sum_{t=0}^{k} \Vert \bm{\psi}^{k+1} - \bm{\psi}^k  \Vert^2  &\leq \sum_{t=0}^{k} \Vert \bm{\psi}^{t+1} - \bm{\psi}^t\Vert^2
	\end{align}

	Adding the above two gives
	\begin{equation}
		\Vert \bm{\psi}^{k+1} - \bm{\psi}^k  \Vert^2 \leq \frac{1}{k+1}\Vert \bm{\psi}^\star - \bm{\psi}^0  \Vert^2,
	\end{equation}
	which concludes the
	proof of Proposition 4.1

	\textbf{Remarks 1} (sharp rate): We can easily show that \eqref{eq_rate1} where $ L\in (0,1) $ is a shaper rate compared to \eqref{rate_basic} where $ L=1 $.
	
	To see this, rewrite \eqref{eq_rate1} into
	\begin{equation}
		\Vert \bm{\psi}^{k+1} - \bm{\psi}^k  \Vert^2 \leq \frac{1}{\sum_{t=0}^{k} (\frac{L}{2-L})^{-t}}	\Vert \bm{\psi}^\star - \bm{\psi}^0  \Vert^2. 
	\end{equation}
	Given $ L\in (0,1) $, we have $ \frac{L}{2-L} < 1$. Therefore, 
	\begin{equation}
		\sum_{t=0}^{k} (\frac{L}{2-L})^{-t} > \sum_{t=0}^{k} 1^{-t} = k+1
	\end{equation}
	Hence, the constant in bound \eqref{eq_rate1} is smaller than that of \eqref{rate_basic}, which implies a shaper rate.

	\textbf{Remarks 2} (sharpness examples): Even when $ L $ is very close to 1, bound \eqref{eq_rate1} could be much shaper than \eqref{rate_basic}. For example,
	(i) let $ L = 0.99,  k = 20 $,  bound \eqref{eq_rate1} has a constant roughly 0.0387, while bound \eqref{rate_basic} has a constant roughly 0.0476; 
	(ii) let $ L = 0.99,  k = 100 $,  bound \eqref{eq_rate1} has a constant roughly 0.0031, while bound \eqref{rate_basic} has a constant roughly  0.0099.

	\subsection{proof of Lemma 4.1}
	Let $ (\bm{x}^\star,\bm{z}^\star,\bm{\lambda}^\star) $ be the optimal solution of ADMM solving problem $ f(\bm{x}) + g (\bm{x}) $.
	Then, by the updating rule of ADMM iterates, we have 
	\begin{equation}
		\bm{z}^\star =\,\, \textbf{Prox}_{g}(\bm{x}^\star + \bm{\lambda}^\star),
	\end{equation}
	which implies
	\begin{equation}
		\bm{x}^\star + \bm{\lambda}^\star - \bm{z}^{\star} \in \partial g(\bm{z}^\star).
	\end{equation}
	Substituting relation $ \bm{x}^\star = \bm{z}^\star $ to above yields
	\begin{equation}\label{lam_gradient}
		\bm{\lambda}^\star  \in \partial g(\bm{x}^\star).
	\end{equation}
	
	Since $ \bm{x}^\star $ is the optimal solution if and only if 
	\begin{equation}
		0 \in \partial f(\bm{x}^\star) + \partial g(\bm{x}^\star). 
	\end{equation} 
	Let $ \widetilde{\nabla} f(\bm{x}^\star) $, $ \widetilde{\nabla} g(\bm{x}^\star) $ be the actual choice of gradients that satisfy the above condition. 
	We obtain   
	\begin{equation}
		0 =  \widetilde{\nabla} f(\bm{x}^\star)  + \widetilde{\nabla} g(\bm{x}^\star).
	\end{equation}
	Invoking \eqref{lam_gradient} which states $ \bm{\lambda}^\star  =  \widetilde{\nabla} g(\bm{x}^\star) $, we arrive at
	\begin{equation}
		\bm{\lambda}^\star  = -\widetilde{\nabla} f(\bm{x}^\star),
	\end{equation}
	which concludes the proof.

	\subsection{proof of Proposition 4.2}
	Here, we aim to prove the following optimal matrix step-size result for ADMM:
	\begin{equation}
		\bm{M}^\star = \text{Diag}(\bm{d}^\star), \quad
		d_i^\star = \begin{cases}
			\vert\lambda_i^\star/x_i^\star\vert  &  \quad{x}_i^\star \neq 0, \\
			+\infty & 	\quad{x}_i^\star = 0,
		\end{cases} 
	\end{equation}
	
	To start, let us note that a diagnal matrix step-size $ \bm{M} \eqdef \text{Diag}(\bm{d}) \succ 0$ can be decomposed into 
	\begin{equation}
		\bm{M} =  \bm{D}\bm{D}^\text{T},
	\end{equation}
	where $ \bm{D} \eqdef\text{Diag}(\bm{z}) \bm{Q}$,  $ \bm{z}^\text{T}\bm{z} \eqdef \bm{d} $, and where $ \bm{Q} $ is an orthonormal matrix, i.e., $ \bm{Q}\bm{Q}^\text{T} = \bm{I}$.

	The  ADMM fixed-point from the new proximal operator is given by $ \bm{\psi}^\star = \bm{D}{\bm{x}}^\star + (\bm{D}^{-1})^\text{T}{\bm{\lambda}}^\star$.
	The optimization step-size is therefore determined by the following problem: 
	\begin{equation}
		\underset{\bm{D}}{\text{argmin}} \,\,  \Vert\bm{D}\bm{x}^\star +  (\bm{D}^{-1})^\text{T}\bm{\lambda}^\star \Vert^2 = 
		\underset{\bm{D}}{\text{argmin}} \,\,  \Vert\bm{D}\bm{x}^\star\Vert^2 + \Vert (\bm{D}^{-1})^\text{T}\bm{\lambda}^\star \Vert^2. 
	\end{equation}
	It follows that
	\begin{align}
		\Vert\bm{D}\bm{x}^\star\Vert^2 + \Vert (\bm{D}^{-1})^\text{T}\bm{\lambda}^\star \Vert^2
		=\, &   \Vert \text{Diag}(\bm{z})\bm{Q}\bm{x}^\star\Vert^2 + \Vert \bm{Q}(\text{Diag}(\bm{z}))^{-1} \bm{\lambda}^\star \Vert^2\\
		=\, &   \Vert \text{Diag}(\bm{z})\bm{x}^\star\Vert^2 + \Vert (\text{Diag}(\bm{z}))^{-1}\bm{\lambda}^\star \Vert^2\nonumber\\
		=\, &   \sum_{i}\vert z_i {x}^\star_i  \vert^2 +  \sum_{i} \vert{\lambda}^\star_i / z_i  \vert^2 \nonumber\\
		=\, &  \sum_{i} d_i\vert{x}^\star_i\vert^2 +  \sum_{i} \vert{\lambda}^\star_i \vert^2/d_i.\nonumber\label{N_problems}  
	\end{align}
	Let us note that the above sum is minimized if and only if it is element-wisely minimized. That said, the $ i $-th optimal entry $ d_i^\star $ can be found via
	\begin{equation}
		\underset{d_i>0}{\text{argmin}} \,\, d_i\vert{x}^\star_i\vert^2 + \vert{\lambda}^\star_i \vert^2/d_i,
	\end{equation}
	where the minimum is obtained if and only if $  d_i\vert{x}^\star_i\vert^2 = \vert{\lambda}^\star_i \vert^2/d_i $. This gives
	\begin{equation}
		d_i^\star = \frac{\vert\lambda^\star_i\vert}{\vert{x}^\star_i\vert}. 
	\end{equation}
	Let us note that the denominator $ {x}^\star_i $ could be zero, which corresponds to vector $ \bm{x} $ being sparse. In which case, it yields $ d_i^\star =   +\infty$. The objective value is hence given by $ d_i\vert{x}^\star_i\vert^2 + \vert{\lambda}^\star_i \vert^2/d_i =  0 $, where we treat $ +\infty $ times $ 0 $ to be zero. In practice, one only needs to set $ d_i $ being relatively large, and the term would be negligible.

	\subsection{proof of bounds of X in Section 5.1}
	Here, we aim to prove the following relation as in eq. (5.3):
	\begin{equation}
		\sqrt{N+1} \leq \Vert\bm{X}^\star\Vert \leq N+1.
	\end{equation}
	
	To show the above relation, recall the BQP is given by
	\begin{align}
		\underset{\bm{X}}{\text{minimize}}\quad  & \langle \bm{X}, \bm{C} \rangle \nonumber\\
		\text{subject\,to}\quad 
		&  \text{diag}(\bm{X})= \bm{1}_{(N+1)\times 1} \nonumber\\
		&  \bm{X} \succeq 0.
	\end{align}
	The first constraint states that $ \text{trace}(\bm{X}^\star) = N+1 $. That is, 	the sum of eigenvalues of $ \bm{X} $ equals to $ N+1 $. The second constraint states that $ \bm{X}^\star $ is positive semidefinite, i.e., it has non-negative eigenvalues. 
	
	Then, we have that $ \Vert\bm{X}^\star\Vert $ is at most $ N+1 $, which corresponds to $ \bm{X}^\star $ contains a single eigenvalue of $ N+1 $; it is at least $ \sqrt{N+1}  $,  which corresponds to $ \bm{X}^\star $ contains $ N+1 $ equal eigenvalues, each of value $ 1 $. The proof is now concluded.

	\subsection{proof of bounds of X in Section 5.2}
	Here, we aim to prove the following relation as in eq. (5.5):
	\begin{equation}
		(N+1)\sqrt{K} m_{avg}  \le \Vert\bm{X}^\star\Vert < (N+1)K m_{avg} ,
	\end{equation}
	where $ m_{avg} \eqdef \frac{1}{K}\sum_{k=1}^K|c_{k}| $.

	To show the above relation, first recall the SR is given by
	\begin{align}
		\underset{\bm{X}}{\text{minimize}}\quad  & \langle \bm{X}, \bm{C} \rangle   \nonumber\\
		\text{subject\,to}\quad 
		&  x_j =  x_j^\star, \: \forall j\in\Omega  \nonumber\\
		& \bm{X} \succeq 0,
	\end{align}
	where $ \bm{X} \eqdef \left[\begin{array}{cc}\mathcal{T}(\bm{u}) & \bm{x}\\\bm{x}^\text{H} & t\end{array}\right]$, 
	$ \bm{C} \eqdef  \left[\begin{array}{cc} \frac{1}{2N}\bm{I}_N & 0 \\ 0  &  \frac{1}{2} \end{array} \right]$,  and where $ \mathcal{T} $ is a Toeplitz operator.
		
	Due to the presence of the Toeplitz operator $ \mathcal{T} $,  the following Vandermonde decomposition (see e.g., \cite{tang2013compressed}) holds:
	\begin{equation}
		\bm{X}^\star = \sum_{k=1}^K|c_{k}|   \left[\begin{array}{c}\bm{a}(\tau_{k}) \\ 1 \end{array}\right] \left[\begin{array}{c}\bm{a}(\tau_{k}) \\ 1 \end{array}\right]^\text{H},
	\end{equation} 
	where
	$ \bm{a}(\tau_{k}) \eqdef \left[\begin{array}{ccc} e^{i2\pi(0)\tau_{k}} & \cdots & e^{i2\pi(N-1)\tau_{k}}\end{array}\right]^{\text{T}} $.
	
	By definition, a Toeplitz matrix has all its diagonal elements being the same. Then, following from above, we have
	\begin{equation}
	\text{trace}(\bm{X}^\star) = (N+1) \sum_{k=1}^K|c_{k}|.	
	\end{equation}
	Meanwhile, by the second constraint from SR formulation, we have $ \bm{X}^\star $ is  positive semidefinite, i.e.,   it has non-negative eigenvalues.

	For SR, the number of eigenvalues of $ \bm{X}^\star $ is fixed to $ K $, i.e., $  \text{rank}(\bm{X}^\star) = K$. Then, the upper bound (but not attainable) of $ \Vert\bm{X}^\star\Vert $ is given by 
	$ (N+1) \sum_{k=1}^K|c_{k}| $, which corresponds to $ \bm{X}^\star $ contains a single eigenvalue of $ (N+1) \sum_{k=1}^K|c_{k}| $; 
	the lower bound is given by $ (N+1)/\sqrt{K} \sum_{k=1}^K|c_{k}|  $,  which corresponds to $ \bm{X}^\star $ contains $ K $ equal eigenvalues, each of value $ (N+1)/K \sum_{k=1}^K|c_{k}| $. 

	Substituting  definition $ m_{avg} \eqdef \frac{1}{K}\sum_{k=1}^K|c_{k}| $  concludes the proof.

		\subsection{proof of Lagrange analysis in Section 5.1} \label{appendix_Lag1}
	The Boolean quadratic program  can be written  as	
	\begin{align}
		\underset{\bm{X}}{\text{minimize}}\quad  & \langle \bm{X}, \bm{C} \rangle \nonumber\\
		\text{subject\,to}\quad 
		&  \text{diag}(\bm{X})= \bm{1}_{(N+1)\times 1} \nonumber\\
		&  \bm{X} \succeq 0. 
	\end{align}
	We can form the following Lagrangian:
	\begin{equation}
		\mathcal{L} = \langle \bm{X}, \bm{C} \rangle + \langle \bm{\mu}, \text{diag}(\bm{X}) - \bm{1}_{(N+1)\times 1} \rangle + \langle \bm{\Lambda}, \bm{X} \rangle,
	\end{equation}
	where $ \text{diag}(\cdot) $ denotes the operation of taking the diagonal elements of a matrix input.
	
	Minimize the above w.r.t. $ \bm{X} $ gives an optimality condition
	\begin{equation}\label{BQP_dual}
		\bm{C} + \text{Diag}(\bm{\mu}^\star) + \bm{\Lambda}^\star = 0,
	\end{equation}
	where $ \text{Diag}(\cdot) $ denotes the operation of creating a diagonal matrix of a vector input. 
	The proof is now concluded.

	\subsection{proof of Lagrange analysis in Section 5.2} \label{appendix_Lag2}
	The super-resolution problem  can be written  as	
	\begin{align}
		\underset{\bm{X}}{\text{minimize}}\quad  & \langle \bm{X}, \bm{C} \rangle \nonumber\\
		\text{subject\,to}\quad 
		&  \bm{X}_\Omega = \bm{X}^\star_\Omega \nonumber\\
		&  \bm{X} \succeq 0,
	\end{align}
	where the entries of $ \bm{X}_\Omega $ are zeros if they are not in the support set $ \Omega $.
	
	We can form the following Lagrangian:
	\begin{equation}
		\mathcal{L} = \langle \bm{X}, \bm{C} \rangle + \langle \bm{M},  \bm{X}_\Omega - \bm{X}^\star_\Omega \rangle + \langle \bm{\Lambda}, \bm{X} \rangle.
	\end{equation}

	Minimize the above w.r.t. $ \bm{X} $ gives an optimality condition
	\begin{equation}
		\bm{C} + \bm{M}_\Omega + \bm{\Lambda}^\star = 0,
	\end{equation}
	where the entries of $ \bm{M}_\Omega $ are zeros if they are not in the support set $ \Omega $.
	The proof is now concluded.

	\section{Additional experiments}
	In this section, we present some additional experiments results with different data settings. For reproduction purpose, unless specified, we fix the MATLAB random number generator using command: rng(`default'). The underlying best fixed step-size choice $\gamma_{best}$ is obtained by exhaustive search. This value may contain ambiguity since there may exist different choices with  same iteration number. 
	  We calculate the MSE using the `ground truth' generated from the solver CVX \cite{grant11} in a best precision mode, and the threshold is set to $ 10^{-10} $.

	\subsection{BQP}
	
	For the Boolean quadratic program, we have $ \bm{A}\in\mathbb{R}^{K\times N} $, $ \bm{b}\in\mathbb{R}^{K\times 1} $. The pre-fixed step-size is given by	$ \gamma_{prefix} = {\Vert\bm{C}\Vert}/(N+1)$; the adaptive step-size is given by $ \gamma_{adap} = \Vert\bm{\Lambda}^k\Vert/\Vert\bm{X}^k\Vert$ with initialization set to be 1; the theoretical optimal step-size  is given by $ \gamma_{theo} = \Vert\bm{\Lambda}^\star\Vert/\Vert\bm{X}^\star\Vert$.

	\begin{table}[H]
		\caption{ performance of different step-sizes $\gamma$ in BQP}
		\begin{tabular}{ccccllllllc}
			\toprule
			&     &	&		&\multicolumn{2}{c}{$\gamma_{best}$}&\multicolumn{2}{c}{$\gamma_{prefix}$}
			&\multicolumn{2}{c}{$\gamma_{theo}$}	&$\gamma_{adap}$\\
			\cmidrule(r){5-6}\cmidrule(r){7-8}\cmidrule(r){9-10}\cmidrule(r){11-11}
			K    & N   	&$ \bm{A} $			&$\bm{b} $		& val		& itr		& val		& itr	& val		& itr		& itr\\
			\midrule
			10   & 10    	&$ \mathcal{N}(0, 1e^{-4}) $ &$ \mathcal{N}(0, 1) $& $ 1.54e^{-4} $& 116& $ 1.16e^{-4} $& 149 & $ 2.56e^{-4} $& 197		& 209\\
			10   & 10    	&$ \mathcal{N}(0, 1e^{-3}) $ &$ \mathcal{N}(0, 1) $& $ 1.19e^{-3} $& 129& $ 1.16e^{-3} $& 139 & $ 2.56e^{-3} $& 212		& 250\\
			10   & 10       &$ \mathcal{N}(0, 1e^{-2}) $ &$ \mathcal{N}(0, 1) $& $ 1.04e^{-2} $& 168
			& $ 1.17e^{-2} $& 206 & $ 2.56e^{-2} $& 374		& 431\\
			10   & 10      &$ \mathcal{N}(0, 1e^{-1}) $ &$ \mathcal{N}(0, 1) $& $ 1.09e^{-1} $& 132
			& $ 1.31e^{-1} $& 141    & $ 2.79e^{-1} $& 311		& 370\\
			10   & 10      &$ \mathcal{N}(0, 1) $ &$ \mathcal{N}(0, 1) $ & $ 2.67 $& 73
			& $ 6.09 $& 159   & $ 7.77 $& 204		& 200\\
			
			20   & 10    	& $ \mathcal{N}(0, 1e^{-4}) $ &$ \mathcal{N}(0, 1) $
			& $ 2.27e^{-4} $& 97	   & $ 1.24e^{-4} $& 173		& $ 2.74e^{-4} $& 162		& 155\\
			
			20   & 10    	& $ \mathcal{N}(0, 1e^{-3}) $ &$ \mathcal{N}(0, 1) $
			& $ 2.08e^{-3} $& 100	   & $ 1.24e^{-3} $& 132		& $ 2.74e^{-3} $& 120		& 173\\
			
			20   & 10    	& $ \mathcal{N}(0, 1e^{-2}) $ &$ \mathcal{N}(0, 1) $
			& $ 2.26e^{-2} $& 98	   & $ 1.24e^{-2} $& 173		& $ 2.74e^{-2} $& 132		& 222\\
			
			20   & 10    	& $ \mathcal{N}(0, 1e^{-1}) $ &$ \mathcal{N}(0, 1) $
			& $ 9.56e^{-2} $& 138	   & $ 1.51e^{-1} $& 200		& $ 2.80e^{-1} $& 377		& 458\\
			
			20   & 10    	& $ \mathcal{N}(0, 1) $ &$ \mathcal{N}(0, 1) $
			& $ 3.51 $  & 83   & $ 8.69 $& 111		& $ 9.54 $ & 118		& 168\\
			
			10   & 10    	& $ \mathcal{N}(0, 1) $ &$ \mathcal{N}(0, 1e^{4}) $
			& $ 1.54e^{4} $ & 116	   & $ 1.16e^{4} $& 149	& $ 2.56e^{4} $& 197		& 207\\
			
			10   & 10    	& $ \mathcal{N}(0, 1) $ &$ \mathcal{N}(0, 1e^{3}) $
			& $ 1.06e^{3} $ & 129	   & $ 1.16e^{3} $& 139	& $ 2.56e^{3} $& 212		& 200\\
			
			10   & 10    	& $ \mathcal{N}(0, 1) $ &$ \mathcal{N}(0, 1e^{2}) $
			& $ 8.34e^{2} $ & 165	   & $ 1.17e^{2} $& 206		& $ 2.56e^{2} $& 374		& 338\\
			
			10   & 10    	& $ \mathcal{N}(0, 1) $ &$ \mathcal{N}(0, 1e^{1}) $
			& $ 1.10e^{1} $ & 132	   & $ 1.31e^{1} $& 141		& $ 2.79e^{1} $& 311		& 274\\
			
			20   & 10    	& $ \mathcal{N}(0, 1) $ &$ \mathcal{N}(0, 1e^{4}) $
			& $ 2.27e^{4} $& 97	   & $ 1.24e^{4} $& 173		& $ 2.74e^{4} $& 162		& 165\\
			
			20   & 10    	& $ \mathcal{N}(0, 1) $ &$ \mathcal{N}(0, 1e^{3}) $
			& $ 2.41e^{3} $& 97	   & $ 1.24e^{3} $& 132		& $ 2.74e^{3} $& 120		& 138\\
			
			20   & 10    	& $ \mathcal{N}(0, 1) $ &$ \mathcal{N}(0, 1e^{2}) $
			& $ 2.26e^{2} $& 98	   & $ 1.24e^{2} $& 173	& $ 2.74e^{2} $& 132		& 147\\
			
			20   & 10    	& $ \mathcal{N}(0, 1) $ &$ \mathcal{N}(0, 1e^{1}) $
			& $ 9.50 $& 138	   & $ 1.50e^{1} $& 199	& $ 2.80e^{1} $& 376		& 383\\
			\bottomrule
		\end{tabular}
	\end{table}
	
	\subsection{SR}

	\begin{table}[H]
		\caption{ performance of different step-sizes $\gamma$ in SR}
		\begin{tabular}{ccccllllllc}
			\toprule
			&     &	&		&\multicolumn{2}{c}{$\gamma_{best}$}&\multicolumn{2}{c}{$\gamma_{prefix}$}
			&\multicolumn{2}{c}{$\gamma_{theo}$}	&$\gamma_{adap}$\\
			\cmidrule(r){5-6}\cmidrule(r){7-8}\cmidrule(r){9-10}\cmidrule(r){11-11}
			K    & N   	& $ \vert \Omega \vert/N$			& $\bm{c} $		& val		& itr		& val		& itr	& val		& itr		& itr\\
			\midrule
			2   & 20    	& 0.5 &$ \mathcal{N}(0, 1e^{-2}) $
			& $ 5.89e^{-1} $& 111	   & $ 2.16 $ & 134	& $ 1.28 $& 115		& 88\\
			
			2   & 20    	& 0.5 &$ \mathcal{N}(0, 1e^{-1}) $
			& $ 6.37e^{-2} $& 140	   & $ 2.16e^{-1} $& 203	& $ 1.28e^{-1} $& 185		& 157\\
			
			2   & 20    	& 0.5 &$ \mathcal{N}(0, 1) $
			& $ 6.21e^{-3} $& 166	   & $ 2.16e^{-2} $& 273	& $ 1.28e^{-2} $& 233		& 236\\
			
			2   & 20    	& 0.5 &$ \mathcal{N}(0, 1e^1) $
			& $ 6.08e^{-4} $& 219	   & $ 2.16e^{-3} $& 343	& $ 1.28e^{-3} $& 303		& 317\\
			
			2   & 20    	& 0.5 &$ \mathcal{N}(0, 1e^2) $
			& $ 6.28e^{-5} $& 255	   & $ 2.16e^{-4} $& 435	& $ 1.28e^{-4} $& 351		& 407\\

			2   & 20    	& 0.9 &$ \mathcal{N}(0, 1e^{-2}) $
			& $ 6.08e^{-1} $ & 49	   & $ 2.16 $& 133	& $ 1.12 $& 94		& 104\\
			
			2   & 20    	& 0.9 &$ \mathcal{N}(0, 1e^{-1}) $
			& $ 4.16e^{-2} $ & 87	   & $ 2.16e^{-1} $& 183	& $ 1.12e^{-1} $& 143		& 152\\
			
			2   & 20    	& 0.9 &$ \mathcal{N}(0, 1) $
			& $ 3.99e^{-3} $ & 122	   & $ 2.16e^{-2} $& 233	& $ 1.12e^{-2} $& 193		& 185\\
			
			2   & 20    	& 0.9 &$ \mathcal{N}(0, 1e^1) $
			& $ 4.39e^{-4} $ & 154	   & $ 2.16e^{-3} $& 298	& $ 1.12e^{-3} $& 227		& 272\\
			
			2   & 20    	& 0.9 &$ \mathcal{N}(0, 1e^2) $
			& $ 3.75e^{-5} $ & 188	   & $ 2.16e^{-4} $& 363	& $ 1.12e^{-4} $& 292		& 306\\
			
			5   & 20    	& 0.7 &$ \mathcal{N}(0, 1e^{-1}) $
			& $ 3.41e^{-1} $ & 253	   & $ 1.37e^{-1} $& 603	& $ 3.41e^{-1} $& 253		& 251\\
			
			5   & 20    	& 0.7 &$ \mathcal{N}(0, 1) $
			& $ 3.32e^{-2} $ & 316	   & $ 1.37e^{-2} $& 837	& $ 3.41e^{-2} $& 316		& 339\\
			
			5   & 20    	& 0.7 &$ \mathcal{N}(0, 1e^1) $
			& $ 3.27e^{-3} $ & 403	   & $ 1.37e^{-3} $& 1064	& $ 3.41e^{-3} $& 418		& 457\\
			
			10   & 100    	& 0.7 &$ \mathcal{N}(0, 1e^{-1}) $
			& $ 1.97e^{-2} $ & 556	   & $ 1.97e^{-2} $ & 556	& $ 1.79e^{-2} $& 600		& 849 \\
			
			10   & 100    	& 0.7 &$ \mathcal{N}(0, 1) $
			& $ 2.25e^{-3} $ & 779	   & $ 1.97e^{-3} $& 862	& $ 1.79e^{-3} $& 954		& 1042 \\
			
			10   & 100    	& 0.7 &$ \mathcal{N}(0, 1e^1) $
			& $ 2.41e^{-4} $ & 995	   & $ 1.97e^{-4} $& 1202	& $ 1.79e^{-4} $& 1333		& 1784 \\
			\bottomrule
		\end{tabular}
	\end{table}

		For the super-resolution problem, we have $ K $ spikes and their amplitudes are $ \bm{c}\in\mathbb{R}^{K\times 1} $. Also, data $ \bm{x}\in\mathbb{R}^{N\times 1} $. The pre-fixed step-size is given by	$ \gamma_{prefix} = {\Vert\bm{C}\Vert}/((N+1)\sqrt{K}\sigma\sqrt{2/\pi})$; the adaptive step-size is given by $ \gamma_{adap} = \Vert\bm{\Lambda}^k\Vert/\Vert\bm{X}^k\Vert$ with initialization set to be 1; the theoretical optimal step-size  is given by $ \gamma_{theo} = \Vert\bm{\Lambda}^\star\Vert/\Vert\bm{X}^\star\Vert$. Recall $ \Omega $ denotes the support set and its size is denoted as  $ \vert \Omega \vert $. The ratio $ \vert \Omega \vert/ N $ shows the percentage of entries observed (partial observation).

	\begin{remark}[Special cases]
		(i) In the first line, we see that the adaptive step-size choice $\gamma_{adap}$ outperforms the best fixed step-size choice $\gamma_{best}$. Also, the theoretical choice $\gamma_{theo}$ has almost the same iteration number complexity as $\gamma_{best}$. (ii) For the case where $ K=5, N=20, \vert \Omega \vert/N = 0.7 $, the first two settings shows that $\gamma_{theo}$ has the same performance as $\gamma_{best}$. This shows that the underlying best choice is in fact attainable by the theoretical choice $\gamma_{theo}$, which is derived from a convergence rate upper bound.
	\end{remark}

	\subsection{LASSO}
	For the LASSO problem, it handles problem: minimize $  1/2|| \bm{Ax - b} ||_2^2 + \alpha || \bm{x} ||_1 $.
	We  adopt the data settings from  \cite{boyd12}. We also adopt their random number fixing approach, using command: `randn('seed',0); rand('seed',0)'.
	It first generates a sparse ground-truth $ \bm{x}^\star $, where the sparsity density is controlled by factor $ p \in (0,1)$. Then, it generates $ \bm{b} = \bm{A}\bm{x}^\star + \bm{\epsilon} $, where $ \bm{\epsilon} $ can be viewed as a noise term and where $ \bm{A} $ is normalized. 
	$ \alpha $ is fixed to  $ 0.1\Vert \bm{A}^T \bm{b} \Vert_\infty$.

	To handle the non-smooth term $ || \cdot||_1 $, we rewrite the above problem to the following constrained form:
	\begin{align}
		\underset{\bm{x},\bm{z}}{\text{minimize}}\,\,\quad &    1/2|| \bm{Ax - b} ||_2^2 + \alpha || \bm{z} ||_1 ,\nonumber\\
		\text{subject\,to}\:\quad & \:\,\,\,\bm{x} - \bm{z} = 0,
	\end{align}
	where $ \bm{A}\in\mathbb{R}^{K\times N} $, $ \bm{x}\in\mathbb{R}^{N\times 1} $,  $ \bm{b}\in\mathbb{R}^{K\times 1} $, $ \bm{z}\in\mathbb{R}^{N\times 1} $.
	
	This implies that the fixed-point is given by  $ \sqrt{\gamma}\bm{x}^\star  + \bm{\lambda}^\star /\sqrt{\gamma} $. Hence, we set the theoretical optimal step-size  to be $ \gamma_{theo} = \Vert\bm{\lambda}^\star\Vert/\Vert\bm{x}^\star\Vert$ and the adaptive step-size to be $ \gamma_{adap} = \Vert\bm{\lambda}^k\Vert/\Vert\bm{x}^k\Vert$ with initialization set to be 1.
	\begin{table}[H]
		\caption{ performance of different step-sizes $\gamma$ in LASSO}
		\begin{tabular}{ccccllclllc}
			\toprule
			&     &	&		&\multicolumn{2}{c}{$\gamma_{best}$}&\multicolumn{2}{c}{$\gamma_{1}$}
			&\multicolumn{2}{c}{$\gamma_{theo}$}	&$\gamma_{adap}$\\
			\cmidrule(r){5-6}\cmidrule(r){7-8}\cmidrule(r){9-10}\cmidrule(r){11-11}
			K    & N   	& $ p $			& $ \bm{\epsilon} $		& val		& itr		& val		& itr	& val		& itr		& itr\\
			\midrule
			20   & 100    	& 0.1 &$ \mathcal{N}(0, 1e^{-2}) $
			& $ 8.96e^{-1} $& 115	 &$ 1 $& 165	  & $ 6.61e^{-1} $  	&132 	& 136\\
			
			20   & 100    	& 0.1 &$ \mathcal{N}(0, 1e^{-1}) $
			& $ 8.13e^{-1} $& 112	 &$ 1 $& 118	  & $ 6.11e^{-1} $  	&136 	& 147\\
			
			20   & 100    	& 0.1 &$ \mathcal{N}(0, 1) $
			& $ 3.26e^{-1} $& 183	 &$ 1 $& 1030	  & $ 5.54e^{-1} $  	&531 	& 534\\
			
			20   & 100    	& 0.1 &$ \mathcal{N}(0, 1e^{1}) $
			& $ 5.65e^{-1} $& 139	 &$ 1 $& 539	  & $ 5.89e^{-1} $  	&213 	& 228\\
			
			20   & 100    	& 0.1 &$ \mathcal{N}(0, 1e^{2}) $
			& $ 6.13e^{-1} $& 176	 &$ 1 $& 314	  & $ 5.75e^{-1} $  	&194 	& 203\\

			50   & 100    	& 0.1 &$ \mathcal{N}(0, 1e^{-1}) $
			& $ 1.11 $& 32	 &$ 1 $& 33	  & $ 4.38e^{-1} $  	&64 	& 66\\
			
			50   & 100    	& 0.1 &$ \mathcal{N}(0, 1) $
			& $ 4.94e^{-1} $& 66	 &$ 1 $& 310	  & $ 3.81e^{-1} $  	&71 	& 68\\
			
			50   & 100    	& 0.1 &$ \mathcal{N}(0, 1e^{1}) $
			& $ 3.70e^{-1} $& 105	 &$ 1 $& 428	  & $ 4.06e^{-1} $  	&131 	& 135\\
			
			50   & 100    	& 0.1 &$ \mathcal{N}(0, 1e^{2}) $
			& $ 3.53e^{-1} $& 110	 &$ 1 $& 372	  & $ 4.13e^{-1} $  	&122 	& 125\\

			20   & 100    	& 0.6 &$ \mathcal{N}(0, 1e^{-1}) $
			& $ 8.81e^{-1} $& 92	 &$ 1 $& 102	  & $ 6.84e^{-1} $  	&102 	& 107\\
			
			20   & 100    	& 0.6 &$ \mathcal{N}(0, 1) $
			& $ 5.21e^{-1} $& 148	 &$ 1 $& 473	  & $ 7.11e^{-1} $  	&302 	& 317\\
			
			20   & 100    	& 0.6 &$ \mathcal{N}(0, 1e^{1}) $
			& $ 3.69e^{-1} $& 306	 &$ 1 $& 582	  & $ 4.88e^{-1} $  	&332 	& 351\\
			
			50   & 100    	& 0.6 &$ \mathcal{N}(0, 1e^{-1}) $
			& $ 4.38e^{-1} $& 69	 &$ 1 $& 291	  & $ 6.17e^{-1} $  	&161 	& 167\\
			
			50   & 100    	& 0.6 &$ \mathcal{N}(0, 1) $
			& $ 2.12e^{-1} $& 127	 &$ 1 $& 736	  & $ 4.09e^{-1} $  	&280 	& 288\\
			
			50   & 100    	& 0.6 &$ \mathcal{N}(0, 1e^{1}) $
			& $ 3.44e^{-1} $& 87	 &$ 1 $& 351	  & $ 3.75e^{-1} $  	&94 	& 97\\

			50   & 300    	& 0.6 &$ \mathcal{N}(0, 1) $
			& $ 8.05e^{-1} $& 121	 &$ 1 $& 201	  & $ 8.52e^{-1} $  	&140 	& 145\\
			
			100   & 300    	& 0.6 &$ \mathcal{N}(0, 1) $
			& $ 3.90e^{-1} $& 116	 &$ 1 $& 349	  & $ 5.50e^{-1} $  	&156 	& 171\\
			
			200   & 300    	& 0.6 &$ \mathcal{N}(0, 1) $
			& $ 3.59e^{-1} $& 57	 &$ 1 $& 160	  & $ 3.76e^{-1} $  	&57 	& 60\\
			\bottomrule
		\end{tabular}
	\end{table}

	\pagebreak
	\subsection{LAD} 
	For the Least absolute deviations (LAD), it handles problem: minimize  $ 1/2|| \bm{Ax} - \bm{b} ||_1 $.
	We  adopt the type of data settings from  \cite{boyd12}. We also adopt their random number fixing approach, using command: `randn('seed',0); rand('seed',0)'.	
	It first randomly generates a ground-truth $ \bm{x}^\star $ from $ \mathcal{N}(0, 10) $ and $ \bm{A} $ from $ \mathcal{N}(0, 1) $.
	Then, $ \bm{b}^\star = \bm{A}\bm{x}^\star $, and some randomly chosen entries recorded by a set $\Omega$ is added by a noise term  $ \bm{\epsilon} $, which gives $ \bm{b}_\Omega = \bm{b}^\star_\Omega + \bm{\epsilon} $. The size of the set is denoted by $ \vert \Omega \vert $.
	Also, $ \bm{A}\in\mathbb{R}^{K\times N} $, $ \bm{x}\in\mathbb{R}^{N\times 1} $, $ \bm{b}\in\mathbb{R}^{K\times 1} $.
	
	To solve it via ADMM, we introduce an auxiliary variable $ \bm{z}^k \eqdef \bm{Ax}^k - \bm{b}$. This implies that the ADMM fixed-point is given by  $ \sqrt{\gamma}\bm{Ax}^\star  + \bm{\lambda}^\star /\sqrt{\gamma} $.
	Hence, we set the theoretical optimal step-size  to be $ \gamma_{theo} = \Vert\bm{\lambda}^\star\Vert/\Vert\bm{Ax}^\star\Vert$ and the adaptive step-size to be $ \gamma_{adap} = \Vert\bm{\lambda}^k\Vert/\Vert\bm{Ax}^k\Vert$ with initialization set to be 1.
	\begin{table}[H]
		\caption{ performance of different step-sizes $\gamma$ in LAD}
		\begin{tabular}{ccccllclllc}
			\toprule
			&     &	&		&\multicolumn{2}{c}{$\gamma_{best}$}&\multicolumn{2}{c}{$\gamma_{1}$}
			&\multicolumn{2}{c}{$\gamma_{theo}$}	&$\gamma_{adap}$\\
			\cmidrule(r){5-6}\cmidrule(r){7-8}\cmidrule(r){9-10}\cmidrule(r){11-11}
			K    & N   	& $ \vert \Omega \vert $		& $ \bm{\epsilon} $		& val		& itr		& val		& itr	& val		& itr		& itr\\
			\midrule
			500   & 100    	& 20 &$ \mathcal{N}(0, 1e^{2}) $
			& $ 2.75e^{-2} $& 21	 &$ 1 $& 36	  & $ 4.98e^{-3} $  	&40 	& 24\\
			
			500   & 100    	& 20 &$ \mathcal{N}(0, 1e^{3}) $
			& $ 2.09e^{-3} $& 25	 &$ 1 $& 210	  & $ 5.21e^{-3} $  	&25 	& 27\\
			
			500   & 100    	& 20 &$ \mathcal{N}(0, 5e^{3}) $
			& $ 4.21e^{-4} $& 28	 &$ 1 $& 997	  & $ 5.38e^{-3} $  	&30 	& 28\\
			
			500   & 100    	& 20 &$ \mathcal{N}(0, 1e^{4}) $
			& $ 1.21e^{-3} $& 29	 &$ 1 $& $ 1.98e^3 $	  & $ 5.39e^{-3} $  	&36 	& 29\\

			1000   & 100    	& 20 &$ \mathcal{N}(0, 1e^{2}) $
			& $ 1.61e^{-1} $& 13	 &$ 1 $& $ 17 $	  & $ 4.82e^{-3} $  	&58 	& 54\\
			
			1000   & 100    	& 20 &$ \mathcal{N}(0, 1e^{3}) $
			& $ 1.91e^{-2} $& 16	 &$ 1 $& $ 93 $	  & $ 5.24e^{-3} $  	&17 	& 18\\
			
			1000   & 100    	& 20 &$ \mathcal{N}(0, 1e^{4}) $
			& $ 3.97e^{-4} $& 18	 &$ 1 $& $ 854 $	  & $ 5.34e^{-3} $  	&21 	& 20\\

			500   & 100    	& 90 &$ \mathcal{N}(0, 1e^{2}) $
			& $ 2.23e^{-2} $& 36	 &$ 1 $& 78	  & $ 5.90e^{-3} $  	&74 	& 83\\
			
			500   & 100    	& 90 &$ \mathcal{N}(0, 1e^{3}) $
			& $ 8.58e^{-3} $& 43	 &$ 1 $& 557	  & $ 6.08e^{-3} $  	&45 	& 46\\
			
			500   & 100    	& 90 &$ \mathcal{N}(0, 1e^{4}) $
			& $ 2.46e^{-4} $& 50	 &$ 1 $& $ 5.35e^3 $	  & $ 6.19e^{-3} $  	&75 	& 51\\
			
			500   & 100    	& 200 &$ \mathcal{N}(0, 1e^{2}) $
			& $ 4.2e^{-2} $& 93	 &$ 1 $& 450	  & $ 7.80e^{-3} $  	&257 	& 310\\
			
			500   & 100    	& 200 &$ \mathcal{N}(0, 1e^{3}) $
			& $ 8.63e^{-3} $& 121	 &$ 1 $& $ 4.07e^3 $	  & $ 7.89e^{-3} $  	&140 	& 118\\
			
			500   & 100    	& 200 &$ \mathcal{N}(0, 1e^{4}) $
			& $ 9.29e^{-4} $& 145	 &$ 1 $& $ 4.02e^4 $	  & $ 7.91e^{-3} $  	&425 	& 182\\	
			\bottomrule
		\end{tabular}
	\end{table}

	\subsection{QP} 
		For the quadratic programming (QP), it handles problem: minimize $ 1/2 || \bm{Ax - b} ||_2^2$, s.t.  $a \leq \bm{x} \leq b $.
	We  adopt the type of data settings from  \cite{boyd12} except we do not intentionally promote a well-conditioned structure by manipulating the eigenvalues (we use completely random setting). We also adopt their random number fixing approach, using command: `randn('state', 0);rand('state', 0)'.
	All data is randomly generated from i.i.d Gaussian distribution, and $ \bm{A}\in\mathbb{R}^{N\times N} $, $ \bm{x}\in\mathbb{R}^{N\times 1} $, $ \bm{b}\in\mathbb{R}^{N\times 1} $.
	
	To solve it via ADMM, we introduce an auxiliary variable $ \bm{z}^k \eqdef  \bm{x}^k $. This implies that the ADMM fixed-point is given by $ \sqrt{\gamma}\bm{x}^\star  + \bm{\lambda}^\star /\sqrt{\gamma} $.
	Hence, we set the theoretical optimal step-size  to be $ \gamma_{theo} = \Vert\bm{\lambda}^\star\Vert/\Vert\bm{x}^\star\Vert$ and the adaptive step-size to be $ \gamma_{adap} = \Vert\bm{\lambda}^k\Vert/\Vert\bm{x}^k\Vert$ with initialization set to be 1.

	\begin{table}[H]
		\caption{ performance of different step-sizes $\gamma$ in QP}
		\begin{tabular}{cccllclllc}
			\toprule
			&     &			&\multicolumn{2}{c}{$\gamma_{best}$}&\multicolumn{2}{c}{$\gamma_{1}$}
			&\multicolumn{2}{c}{$\gamma_{theo}$}	&$\gamma_{adap}$\\
			\cmidrule(r){4-5}\cmidrule(r){6-7}\cmidrule(r){8-9}\cmidrule(r){10-10}
			N   		& $ \bm{A}$			& $\bm{b} $			& val		& itr		& val		& itr	& val		& itr		& itr\\
			\midrule
			50    	& $ \mathcal{N}(0, 1) $ &$ \mathcal{N}(0, 1) $
			& $ 3.68e^{1} $& 31	 &$ 1 $& 873	  & $ 2.55e^{1} $  	&36 	& 43\\
			
			50    	& $ \mathcal{N}(0, 5) $ &$ \mathcal{N}(0, 1) $
			& $ 7.92e^{2} $& 33	 &$ 1 $& $ 2.62e^{4} $	  & $ 7.58e^{2} $  	&35 	& 48\\
			
			50    	& $ \mathcal{N}(0, 1e^1) $ &$ \mathcal{N}(0, 1) $
			& $ 4.39e^{3} $& 31	 &$ 1 $& $ 8.96e^{4} $	  & $ 2.49e^{3} $  	&39 	& 50\\
			
			50    	& $ \mathcal{N}(0, 1) $ &$ \mathcal{N}(0, 5) $
			& $ 4.12e^{1} $& 35	 &$ 1 $& $ 1.19e^{3}  $	  & $ 3.98e^{1} $  	&36 	& 44\\
			
			50    	& $ \mathcal{N}(0, 1) $ &$ \mathcal{N}(0, 1e^1) $
			& $ 5.64e^{1} $& 30	 &$ 1 $& $ 1.45e^{3}  $	  & $ 6.89e^{1} $  	&38 	& 36\\
			
			50    	& $ \mathcal{N}(0, 1) $ &$ \mathcal{N}(0, 5e^1) $
			& $ 1.18e^{2} $& 21	 &$ 1 $& $ 1.88e^{3}  $	  & $ 3.02e^{2} $  	&38 	& 50\\
			
			50    	& $ \mathcal{N}(0, 1) $ &$ \mathcal{N}(0, 1e^2) $
			& $ 2.77e^{2} $& 13	 &$ 1 $& $ 2.07e^{3}  $	  & $ 6.22e^{2} $  	&17 	& 31\\
			
			50    	& $ \mathcal{N}(0, 5) $ &$ \mathcal{N}(0, 5) $
			& $ 7.88e^{2} $& 33	 &$ 1 $& $ 2.60e^{4} $	  & $ 7.40e^{2} $  	&37 	& 53\\

			50    	& $ \mathcal{N}(0, 1e^{-1} ) $ &$ \mathcal{N}(0, 1) $
			& $ 5.65e^{-1} $& 30	 &$ 1 $& 57	  & $ 6.88e^{-1} $  	&38 	& 37\\
			
			50    	& $ \mathcal{N}(0, 1e^{-2} ) $ &$ \mathcal{N}(0, 1) $
			& $ 2.77e^{-2} $& 13	 &$ 1 $& 162	  & $ 6.22e^{-2} $  	&17 	& 17\\
			
			50    	& $ \mathcal{N}(0, 1e^{-3} ) $ &$ \mathcal{N}(0, 1) $
			& $ 1.72e^{-3} $& 7	 &$ 1 $& $ 1.32e^{3}  $	  & $ 6.51e^{-3} $  	&12 	& 14\\
			
			50    	& $ \mathcal{N}(0, 1 ) $ &$ \mathcal{N}(0, 1e^{-1}) $
			& $ 3.14e^{1} $& 33	 &$ 1 $& $ 1.06e^{3}  $	  & $ 3.06e^{1} $  	&34 	& 41\\
			
			50    	& $ \mathcal{N}(0, 1 ) $ &$ \mathcal{N}(0, 1e^{-2}) $
			& $ 3.12e^{1} $& 33	 &$ 1 $& $ 1.06e^{3}  $	  & $ 3.08e^{1} $  	&34 	& 38\\
			
			50    	& $ \mathcal{N}(0, 1 ) $ &$ \mathcal{N}(0, 1e^{-3}) $
			& $ 3.12e^{1} $& 33	 &$ 1 $& $ 1.07e^{3}  $	  & $ 3.09e^{1} $  	&34 	& 40\\
			
			50    	& $ \mathcal{N}(0, 1e^{-1} ) $ &$ \mathcal{N}(0, 1e^{-1}) $
			& $ 3.15e^{-1} $& 33	 &$ 1 $& $ 126  $	  & $ 2.96e^{-1} $  	&37 	& 40\\

			300    	& $ \mathcal{N}(0, 1) $ &$ \mathcal{N}(0, 1) $
			& $ 1.49e^{2} $& 40	 &$ 1 $& $ 5.35e^{3} $	  & $ 1.60e^{2} $  	&40 	& 48\\
			
			300    	& $ \mathcal{N}(0, 5) $ &$ \mathcal{N}(0, 1) $
			& $ 3.91e^{3} $& 39	 &$ 1 $& $ 1.34e^{5} $	  & $ 3.88e^{3} $  	&40 	& 49\\
			
			300    	& $ \mathcal{N}(0, 1e^1) $ &$ \mathcal{N}(0, 1) $
			& $ 1.57e^{4} $& 39	 &$ 1 $& $ 5.36e^{5} $	  & $ 1.54e^{4} $  	&40 	& 50\\
			
			300    	& $ \mathcal{N}(0, 1) $ &$ \mathcal{N}(0, 1e^1) $
			& $ 1.91e^{2} $& 36	 &$ 1 $& $ 5.78e^{3} $	  & $ 2.18e^{2} $  	&36 	& 45\\
			
			300    	& $ \mathcal{N}(0, 1) $ &$ \mathcal{N}(0, 1e^2) $
			& $ 4.72e^{2} $& 29	 &$ 1 $& $ 1.05e^{4} $	  & $ 1.56e^{3} $  	&94 	& 104\\
			
			300    	& $ \mathcal{N}(0, 5) $ &$ \mathcal{N}(0, 5) $
			& $ 3.73e^{3} $& 40	 &$ 1 $& $ 1.33e^{5} $	  & $ 4.01e^{3} $  	&40 	& 52\\
			\bottomrule
		\end{tabular}
	\end{table}

	\begin{remark}[non-unique $\gamma_{best}$]
		As mentioned in the beginning of the section, there exists ambiguity of $\gamma_{best}$ in the sense that it is often not unique. That is,  if we  plot the iteration number complexity against different step-size choices, we would observe a flat bottom. As the experiment in the last line above, although $\gamma_{theo}$ has a very different value compared to $\gamma_{best}$, they share the same iteration number complexity. Meanwhile, it is worth noticing that this ambiguity is severe here mainly due to that the iteration number is too small.  
	\end{remark}

	\subsection{TV} 
	For the Total Variation (TV), it handles problem: minimize  $ (1/2)||\bm{x - b}||_2^2 + \alpha \sum_i |x_{i+1} - x_i| $.
	It is a denoising problem, and we adopt the type of data settings from  \cite{boyd12}.  We will use their random number fixing approach via command: `rand('seed', 0);randn('seed', 0)'. Particularly, $ \bm{b} = \bm{x}^\star + \bm{\epsilon} $, where $ \bm{\epsilon} $ is a noise term.

	TV can be written into the following constrained form:
	\begin{align}
		\underset{\bm{x},\bm{z}}{\text{minimize}}\,\,\quad &    1/2||\bm{x - b}||_2^2 + \alpha \Vert \bm{z} \Vert_1,\nonumber\\
		\text{subject\,to}\:\quad & \:\,\,\,\bm{D}\bm{x} - \bm{z} = 0,
	\end{align}
	where $ \Vert\bm{D}\bm{x}\Vert_1 \eqdef \sum_i |x_{i+1} - x_i|  $, and where $ \bm{D}\in\mathbb{R}^{N\times N} $, $ \bm{x}\in\mathbb{R}^{N\times 1} $,  $ \bm{b}\in\mathbb{R}^{N\times 1} $, $ \bm{z}\in\mathbb{R}^{N\times 1} $.

	Following from above, the ADMM fixed-point is given by $ \sqrt{\gamma}\bm{Dx}^\star  + \bm{\lambda}^\star /\sqrt{\gamma} $.
	Therefore, we set the theoretical optimal step-size  to be $ \gamma_{theo} = \Vert\bm{\lambda}^\star\Vert/\Vert\bm{Dx}^\star\Vert$ and the adaptive step-size to be $ \gamma_{adap} = \Vert\bm{\lambda}^k\Vert/\Vert\bm{Dx}^k\Vert$ with initialization set to be 1. Let us note that $ \alpha\Vert \bm{z} \Vert_1 $ is a denoising term and we tend to choose a slightly larger $\alpha$ when the noise level increases and vice versa. The default choice is set to $\alpha = 5$ when $\epsilon $ subject to $ \mathcal{N}(0,1) $ distribution as in \cite{boyd12}.
	\begin{table}[H]
		\centering
		\caption{ performance of different step-sizes $\gamma$ in TV}
		\begin{tabular}{cccllclllc}
			\toprule
			&     &			&\multicolumn{2}{c}{$\gamma_{best}$}&\multicolumn{2}{c}{$\gamma_{1}$}
			&\multicolumn{2}{c}{$\gamma_{theo}$}	&$\gamma_{adap}$\\
			\cmidrule(r){4-5}\cmidrule(r){6-7}\cmidrule(r){8-9}\cmidrule(r){10-10}
			N   		& $ \bm{\epsilon}$			& $\alpha $			& val		& itr		& val		& itr	& val		& itr		& itr\\
			\midrule
			50    	& $ \mathcal{N}(0, 1) $ & 5
			& $ 4.43 $& 107	 &$ 1 $& 452	  & $ 3.13 $  	&147 	& 149\\
			
			50    	& $ \mathcal{N}(0, 1e^{-1}) $ & 5
			& $ 3.04 $& 59	 &$ 1 $& 397	  & $ 3.10 $  	&81 	& 73\\
			
			50    	& $ \mathcal{N}(0, 1e^{-2}) $ & 5
			& $ 4.18 $& 120	 &$ 1 $& 160	  & $ 3.10 $  	&135 	& 136\\
			
			50    	& $ \mathcal{N}(0, 1e^{-2}) $ & 4
			& $ 1.31 $& 119	 &$ 1 $& 204	  & $ 2.31 $  	&148 	& 145\\
			
			50    	& $ \mathcal{N}(0, 1e^{1}) $ & 5
			& $ 8.76e^{-1} $& 42	 &$ 1 $& 45	  & $ 6.66e^{-1} $  	&51 	& 47\\
			
			50    	& $ \mathcal{N}(0, 5e^{1}) $ & 5
			& $ 3.17e^{-1} $& 25	 &$ 1 $& 62	  & $ 8.05e^{-2} $  	&78 	& 82\\
			
			50    	& $ \mathcal{N}(0, 5e^{1}) $ & 6
			& $ 2.94e^{-1} $& 26	 &$ 1 $& 61	  & $ 9.95e^{-2} $  	&67 	& 68\\

			100    	& $ \mathcal{N}(0, 1) $ & 5
			& $ 6.37 $& 199	 &$ 1 $& 1249	  & $ 4.47 $  	&279 	& 274\\
			
			100    	& $ \mathcal{N}(0, 1e^{-1}) $ & 5
			& $ 6.94 $& 213	 &$ 1 $& 1062	  & $ 3.56 $  	&321 	& 291\\
			
			100    	& $ \mathcal{N}(0, 1e^{-2}) $ & 5
			& $ 7.73 $& 220	 &$ 1 $& 628	  & $ 3.47 $  	&284 	& 274\\
			
			100    	& $ \mathcal{N}(0, 1e^{-2}) $ & 4
			& $ 7.28 $& 219	 &$ 1 $& 613	  & $ 2.64 $  	&306 	& 303\\
			
			100    	& $ \mathcal{N}(0, 1e^{1}) $ & 5
			& $ 8.68e^{-1} $& 43	 &$ 1 $& 45	  & $ 7.70e^{-1} $  	&48 	& 48\\
			
			100    	& $ \mathcal{N}(0, 1e^{2}) $ & 5
			& $ 2.79e^{-1} $& 26	 &$ 1 $& 67	  & $ 3.85e^{-2} $  	&153 	& 162\\
			
			100    	& $ \mathcal{N}(0, 1e^{2}) $ & 6
			& $ 2.88e^{-1} $& 27	 &$ 1 $& 67	  & $ 4.69e^{-2} $  	&130 	& 139\\
			
			200    	& $ \mathcal{N}(0, 1) $ & 5
			& $ 8.06 $& 242	 &$ 1 $& 1918	  & $ 4.59 $  	&419 	& 417\\
			
			200    	& $ \mathcal{N}(0, 0.5) $ & 5
			& $ 1.26e^{1} $& 370	 &$ 1 $& 4808	  & $ 4.34 $  	&1136 	& 1137\\
			
			200    	& $ \mathcal{N}(0, 5) $ & 5
			& $ 1.86 $& 71	 &$ 1 $& 131	  & $ 2.33 $  	&82 	&83\\
			
			300    	& $ \mathcal{N}(0, 1) $ & 5
			& $ 1.21e^{1} $& 300	 &$ 1 $& 3546	  & $ 5.22 $  	&682 	& 688\\
			
			400    	& $ \mathcal{N}(0, 1) $ & 5
			& $ 1.05e^{1} $& 233	 &$ 1 $& 2212	  & $ 5.03 $  	&449 	& 459\\
			
			800    	& $ \mathcal{N}(0, 1) $ & 5
			& $ 1.75e^{1} $& 489	 &$ 1 $& 8325	  & $ 7.21 $  	&1153 	& 1165\\
			\bottomrule
		\end{tabular}
	\end{table}

\end{document}